\RequirePackage{fix-cm}
\documentclass[12pt]{amsart}
\usepackage{hyperref}
\usepackage{graphicx}
\usepackage{amsmath}
\usepackage{amsthm}
\usepackage{amssymb}
\usepackage{amscd}
\usepackage{amsfonts}
\usepackage{amsbsy}
\usepackage{bbold}
\usepackage{bm}% bold math
\usepackage{pgfplots}
\usepackage{tkz-graph}
\usepackage{fixmath}
\usepackage{relsize}
\def\pone#1{\mu_{\ref{fig:ic}#1}}
\def\pthree#1{\mu_{\ref{fig:p3ic}#1}}

\def\interior{\text{int}}
\def\orbit{\text{orbit}}
\def\gr{}

\def\INT{_\text{Merge}}
\def\AP{{\mathcal A}_P}
\def\AC{{\mathcal A}_C}
\def\ALP{{\mathcal A}_{AP}}
\def\cT{{\mathcal T}}

\newcommand{\bigletter}[1]{\mathlarger{#1}}
\newcommand{\bigd}{\bigletter{\vdots}}
\newcommand{\bigO}{\bigletter{N}}

\GraphInit[vstyle = Shade]
\renewcommand*{\VertexBallColor}{green!60}
\renewcommand*{\VertexTextColor}{black}

\tikzset{
 LabelStyle/.style = { rectangle, 
            rounded corners, 
            draw, 
            minimum width = 10em, 
            fill = yellow!50, 
            text = black, },
 VertexStyle/.style = { inner sep=3pt, 
             minimum size = 5pt,
             shape     = \VertexShape,
             ball color   = \VertexBallColor,
             color     = \VertexLineColor,
             inner sep   = \VertexInnerSep,
             outer sep   = 0.5\pgflinewidth,
             text      = \VertexTextColor,
             minimum size  = \VertexSmallMinSize,
             line width   = \VertexLineWidth},
 EdgeStyle/.append style = {->, 
               double=yellow, 
               color=orange}
}
\usepackage{enumitem}
\newcommand{\subscript}[2]{$#1 _ #2$}

\newcommand{\inter}[1]{%
 {\kern0pt#1}^{\mathrm{o}}%
}

\usepackage{color}

\newcommand{\BF}{\boldmath}

\usepackage{enumitem}

\definecolor{jim}{rgb}{1,.4,0}
\definecolor{jim2}{rgb}{0,.6,0}

\newcommand{\allblack}{\color{black}{}}

\newcommand{\ie}{{\it{i.e.}}}
\newcommand{\eg}{{\it{e.g.}}}

\def\J{J}
\def\Jall{\J_{all}}
\def\Jall{J^{int}}
\newcommand{\lmu}{\ell_\mu}

\usepackage{accents}
\newlength{\dhatheight}
\newcommand{\doublehat}[1]{%
    \settoheight{\dhatheight}{\ensuremath{\hat{#1}}}%
    \addtolength{\dhatheight}{-0.35ex}%
    \hat{\vphantom{\rule{1pt}{\dhatheight}}%
    \smash{\hat{#1}}}}

\newcommand{\beq}{\begin{linenomath}\begin{equation}} %begin align
\newcommand{\eeq}{\end{equation}\end{linenomath}} %end align

\def\phi{\varphi}

\def\bZ{\mathbb{Z}}

\def\cA{{\mathcal A}}

\def\cO{{\mathcal O}}% added

\def\cR{{\mathcal R}}
\def\cS{{\mathcal S}}

\newtheorem{definition}{Definition}[section]
\newtheorem{theorem}{Theorem}[section]
\newtheorem*{theoremA}{Attractor Theorem}
\newtheorem*{theoremN}{Discrete Conley Theorem}
\newtheorem*{theoremW}{Non-Wandering Theorem}

\newtheorem*{theoremS}{Singer Theorem}
\newtheorem*{theoremH}{Homterval Theorem}
\newtheorem{plc}{Plc}[section]

\newtheorem{corollary}[plc]{Corollary}
\newtheorem{remark}{Remark}[section]
\newtheorem{proposition}[plc]{Proposition}

\title{
 The \gr graph of the logistic map is a tower
}
\author{Roberto De Leo$^1$ and James A. Yorke$^2$}

\address{$^1$ Department of Mathematics, Howard University, Washington DC 20059, USA}
\address{$^2$ Institute for Physical Science and Technology and the Departments of Mathematics and Physics, University of Maryland College Park, MD 20742, USA}
\email{roberto.deleo@howard.edu}
\email{yorke@umd.edu}
\usepackage{lineno}
\pgfplotsset{compat=1.14} 
\begin{document}
\maketitle

\vspace{1cm}
\centerline{\em In memory of Todd A. Drumm (1961-2020) and Tien-Yien Li (1945-2020).}
\vspace{1.5cm}

\begin{abstract}
% Chaotic attractors, chaotic saddles and periodic orbits are examples of chain-recurrent sets. 
% Under arbitrary perturbations, a trajectory starting from any point in a chain-recurrent set can be steered to any other in that set. 
 The qualitative behavior of a dynamical system can be encoded in a graph.
 Each node of the graph is an equivalence class of chain-recurrent points and 
 there is an edge from node $A$ to node $B$ if, using arbitrary small perturbations, a trajectory starting from any point of $A$ can be steered to any point of $B$. 
 %These graphs never contain cycles.
 In this article we describe the graph of the logistic map. 
 Our main result is that the graph is always a tower, namely there is an edge connecting each pair of distinct nodes. Notice that these graphs never contain cycles.
 If there is an edge from node $A$ to node $B$, the unstable manifold of some periodic orbit in $A$ contains points that eventually map onto $B$.
 For special parameter values, this tower has infinitely many nodes. 
\end{abstract}
\section{Introduction}
Ever since H. Poincar\'e invigorated the field of qualitative Dynamical Systems in late 1800s, one of its main goals has been understanding the qualitative asymptotic behavior of points under a continuous or discrete time evolution. In this article we study some fundamental qualitative aspect of the dynamics of the logistic map $\ell_\mu(x)=\mu x(1-x)$ and we represent our results into a graph. %representing the qualitative dynamics of the map. 

The idea of describing the asymptotics of points in a dynamical system through a graph goes back at least to S.~Smale. 
In Sixites, he observed~\cite{Sma61,Sma67} that, as a byproduct of Morse theory, the flow of the gradient vector field of a Morse function $f$ on a compact manifold $M$ can be encoded into a graph:
%satisfies the following two properties: 
1. Its non-wandering set $\Omega_f\subset M$ consists in just a finite number of fixed points $\Omega_i$, $i=0,\dots,p$; these are the nodes of the graph. 2. The dynamics outside $\Omega_f$ consists just of orbits asymptotic to a fixed point $\Omega_i$ for $t\to\infty$ and a different fixed point $\Omega_j$ for $t\to-\infty$; in this case, we say that there is an edge from $\Omega_j$ to $\Omega_i$.
%Since $f$ is Morse, every critical point belongs to a different level set and so it is possible to give a partial order structure to $\Omega_f$: given two fixed points $x,y$, we say that $x>y$ if there is some orbit of the gradient flow of $f$ that converges to $x$ for $t\to-\infty$ and to $y$ for $t\to\infty$. 
%Note in particular that, in such gradient systems, if $x$ is upstream from $y$ and $y$ from $z$, we can find orbits connecting $x$ to both $y$ and $z$ but we cannot find orbits connecting $y$ to $x$ or $z$ to either $y$ or $x$.

Smale also showed that this idea also applies to the more general case of {\em Axiom-A} diffeomorphisms on compact manifolds if we just replace {\em fixed points} with {\em closed disjoint invariant indecomposable subsets} $\Omega_i$, on each of which $f$ is topologically transitive. Recall that, for a Axiom-A diffeomorphism $f$, the non-wandering set $\Omega_f$ is hyperbolic and the set of periodic points is dense in it.
In this case, the nodes of the graph are the $\Omega_i$ and there is an edge from $\Omega_i$ to $\Omega_j$ if and only if the intersection of the stable manifold of $\Omega_i$ with the unstable manifold of $\Omega_j$ is non-empty. Important examples of Axiom-A diffeomorphisms are Morse-Smale diffeomorphisms (those whose nonwandering set consists in a finite number of hyperbolic periodic orbits) and Anosov diffeomorphisms (those for which the whole manifold is hyperbolic).

%In~\cite{Sma61,Sma65,Sma67} he showed that, in case of what are now called Morse-Smale systems, there is a finite number of non-wandering sets $\Omega_1,\dots,\Omega_n$ (nodes) and every orbit outside of the nodes tends asymptotically to a single one of them forward in time and another single one backwards in time. The flow of the gradient of a Morse function in a compact manifold is a fundamental example of a Morse-Smale system. 

%In case of Morse-Smale systems, ever node is a hyperbolic periodic orbit.
Charles Conley extended this idea so that it can be applied to any kind of discrete or continuous dynamical system~\cite{Con78}. 
One of his key ideas was replacing the non-wandering set by the larger set of chain-recurrent points. 
%In the simplest cases, which includes the case of the logistic map, a point is chain-recurrent if there is a small deformation of the dynamical system that makes the point non-wandering~\cite{SBGC84}. 
Chain-recurrent points can be sorted into closed disjoint invariant sets $N_i$ and Conley was able to prove that the dynamics outside of the $N_i$ is gradient-like, namely every trajectory of points outside the chain-recurrent set represents an edge from some $N_i$ to some $N_j$. 

The first to study the non-wandering set of the logistic map were possibly Smale and Williams~\cite{SW76}. 
They studied the particular case $\mu=3.83$, in which case there is a period-3 orbit attractor, and showed that the non-wandering set is given by the union of the attractor, the fixed repelling point $0$ and a Cantor set on which the map acts as a subshift of finite type. 
This invariant Cantor set is topologically invariant for $\mu$ in the period-3 window (in 
%. That Cantor set as function of $\mu$ is shown in 
Figs.~\ref{fig:full} and~\ref{fig:p3} we show that set painted in red). 
A few years later, a complete description of the structure of the non-wandering set was given by Jonker and Rand for unimodal maps~\cite{JR80} and by van Strien for S-unimodal maps~\cite{vS81}, which includes the case of the logistic map.

Although the structure of the non-wandering set has been known for forty years, no one so far has described the graph of the logistic map. 
The main goal of this article is to provide such a description.
We achieve this by first studying the structure of the chain-recurrent set of the logistic map $\ell_\mu$.
Each node $N$ has a periodic point $p_1(N)$ that is closest point of the node to the critical point $c=0.5$.
We denote by $\rho(N)$ that minimum distance and
say that nodes that have larger $\rho(N)$ values are ``higher'' than nodes with lower values. 
We denote by $J_1(N)$ the interval whose endpoints are $p_1(N)$ and $q_1=1-p_1(N)$ and show that $J_1(N)$ maps into itself under some positive power $k$ of $\ell_\mu$. 
This interval $J_1(N)$ is the first one of a cycle of intervals $\cT(N)=\{J_1(N),\ell_\mu(J_1(N)),\dots,\ell^{k-1}_\mu(J_1(N))\}$, invariant under $\ell_\mu$ and containing the attractor. We call $\cT$ a ``cyclic trapping region''. 
Some example of the intervals $J_1$ associated to different nodes are shown in Figs.~\ref{fig:ic2} and~\ref{fig:p3icb}.

Each cyclic trapping region $\cT$ is accompanied by two structures: (1) the periodic orbit of the point $p_1$; (2) the node $N$ containing $p_1$.
Notice that no points of $N$ are in the interior of the trapping region.
%is the trapping region $\cT(N)$ of a node $N$ with $\rho(N)>0$ and that the intersection of $N$ with the intervals of $\cT(N)$ is exactly the periodic orbit passing through $p_1(N)$.
%is a pairing between nodes with positive distance from $c$ and cyclic trapping regions such that the points in common between each node $N$ and its paired trapping region $\cT(N)$ are exactly $p_1(N)$ and its iterates. 
Using the pairing between trapping regions and nodes, we ultimately show that there is an edge of the graph between each pair of nodes, from the higher to the lower. 
%namely the set of nodes of the graph, and then showing that there is an edge between each pair of nodes. 

We call such a graph a tower. 
%The graph of chain-recurrent nodes never has any cycles. 
The node at the bottom of a tower is attracting, all other ones are repelling. 
In a tower, arbitrarily close to each node $N$, for each lower node $N'$, there are points falling eventually on $N'$.
%any other node that lies ``below'' $N$ in the tower. 
There are values of $\mu$ for which the tower has infinitely many levels. Examples of tower graphs of the logistic map are shown in Fig.~\ref{fig:full} and~\ref{fig:p3}.

The paper is structured as follows. All our results are contained in Sec.~3.2 and the reader is recommended to start reading the article from there and to use the rest as needed. 
In order to keep the article self-consistent we include in Sec.~2, all definitions and properties we use about chain-recurrence and the graph of a dynamical system and, in Sec.~3.1, the main theorems in literature we use to prove our own results of Sec.~3.2.
\section{Chain-Recurrence and the \gr graph}
\label{sec:cr}
%
%\begin{definition}
%We assume in this paper that there is a dynamical system $\Phi$ on a compact metric space $(X,d)$, where $\Phi:X\to X$ is continuous. 
Throughout this paper, by dynamical system we mean a continuous map $\Phi:X\to X$ on a compact metric space $(X,d)$.
%{\bf Definition of compact discrete-time dynamical system.} By a dynamical system on a compact metric space $(X,d)$ we mean a continuous map $\Phi:X\to X$. 
The orbit of a point $x\in X$ under $\Phi$ is the set $\{x,\Phi(x),\Phi^2(x),\dots\}$, where $\Phi^n=\underbrace{\Phi\circ\dots\circ\Phi}_{\text{$n$ times}}$.
%   , we mean a 1-parameter commutative semigroup of continuous maps $\Phi^t$ of a metric space $(X,d)$ into itself. The time parameter $t$ can be either continuous or discrete. 
%\end{definition}
%
%Notice that, in particular, $\Phi^0$ is just the identity map and that, for discrete time, the evolution is given by the composition of a given map $f$ with itself, namely $\Phi^n=\underbrace{f\circ\dots\circ f}_{\text{$n$ times}}$.
% %
% \begin{definition}
%     The forward limit set of a point $x\in X$ under $\Phi$ is the set, denoted by $\omega(x)$, of all accumulation points of the orbit $\Phi^t(x)$.
% \end{definition}
%

%In dynamical systems, periodic orbits play a key role and so do several other classes of invariant sets whose points are, in a sense or another, close to periodic. 
In this article, we discuss extensively two classes of points, which we define below: non-wandering points and chain-recurrent points.

%\begin{definition}[Birkhoff, 1927~\cite{Bir27}]
{\bf Definition of non-wandering point.} A point $x\in X$ is a {\bf non-wandering point} for $\Phi$ if, for every neighborhood $U$ of $x$, there is a $n\geq1$ such that $U\cap\Phi^n(U)\neq\emptyset$.
We denote by $\Omega_\Phi$ the set of all non-wandering points of $\Phi$.
% \end{definition}

\bigskip
Notice that every point of a period-$k$ orbit, that is, every fixed point of $\Phi^k$, is trivially a non-wandering point.

%\begin{definition}[Bowen, 1975~\cite{Bow75}]
{\bf Definition of chain-recurrence.}
    Given two points $p,q\in X$ and an $\varepsilon>0$, we say that there is a {\bf\BF$\varepsilon$-chain}~\cite{Bow75} from $p$ to $q$ if there is a finite sequence of points $p=x_0,x_1,\dots,x_n=q$ on $X$ such that, for $i=0,\dots,n-1$,
    \beq
    d(\Phi(x_i),x_{i+i})\le\varepsilon.
    \label{chain}
    \eeq
%\end{definition}
%
%where $\mbox{dist}(x,y)$ denotes the distance between $x$ and $y$ (see Fig.~\ref{fig:cr}). 
%To our knowledge, $\varepsilon$-chains were introduced in literature by R.~Bowen in 1975~\cite{Bow75}.
%
%\begin{definition}[Conley, 1978~\cite{Con78}]
    %\label{def:equiv}
    We say that {\bf \BF $q$ is downstream from $p$} if, for every $\varepsilon> 0$, there is a $\varepsilon$-chain from $p$ to $q$; equivalently, we say that {\bf \BF $p$ is upstream from $q$}. 
    We write {\BF$p\sim q$} if $p$ is upstream and downstream from $q$, and we say that {\bf \BF$p$ is chain recurrent} if $p\sim p$. 
    
    We denote by {\BF${\mathcal R}_\Phi$} the {\bf chain-recurrent set}~\cite{Con78}, \ie\ the set of all chain-recurrent points of $\Phi$. 
    We call each equivalence class in ${\mathcal R}_\Phi$ a {\bf node}. Hence, if $x$ is in a node $N$, then $y\in N$ if and only if $x\sim y$.
    
    \bigskip
%\end{definition}
%
If a point is non-wandering, it is certainly chain-recurrent. However, the following example shows that not all chain-recurrent points are non-wandering.
%    
%\begin{example}
    % Points on a periodic orbit are chain-recurrent and, if $p$ and $q$ are on the same periodic orbit, then $p\sim q$.
    % Chaotic sets are defined in various ways but a usual requirement is that there is a trajectory that comes arbitrarily close to every point infinitely often. So, if $p$ and $q$ are in a chaotic set, a tiny perturbation of $p$ will land on the dense trajectory and, when it comes sufficiently close to $q$, a second tiny perturbation will push it onto $q$. Hence, $p\sim q$. 
%\end{example}
%\begin{example}
%    \label{ex:z2}
    % Consider the map $\Phi(z)=z^2$ on the Riemann sphere. Then $\Omega_\Phi$ is the union of the two fixed points of $\Phi$, namely 0 and $\infty$, and of the unit circle. Indeed, on the unit circle $\Phi$ restricts to the double angle map, so that the image of any non-trivial interval will eventually cover the whole circle.
%\end{example}
%
%\begin{example}
%    \label{ex:Psi}

    Let $\Psi$
%    (\alpha)=\min\{2\alpha,\pi+\alpha/2\}$ 
    map $[0,2\pi]$ into itself such that $\Psi(0)=0$; $\Psi(2\pi)=2\pi$; and $\alpha<\Psi(\alpha)<2\pi$ for $\alpha\in(0,2\pi)$.
    We identify 0 with $2\pi$ to make $\Psi$ a map from the circle to itself. 
%    Notice that $\Psi(\alpha)$ 
    % Consider the map $\Psi(\alpha)=\min\{2\alpha,\pi+\alpha/2\}$ for $0\leq\alpha\leq 2\pi$ from the unit circle
    % %, with angular coordinate $\alpha$, 
    % to itself, where we identify $2\pi$ with 0.
    % % with the following properties:
    % % %
    % % \begin{enumerate}
    % %     \item $\Psi(0)=0$;
    % %     \item $\Psi(\alpha)\in(\min\{2\alpha,\pi+\alpha/2\},2\pi)$ for all $\alpha\neq0$.
    % % \end{enumerate}
    % %
    Then $\lim_{k\to\infty}\Psi^k(\alpha)=0$ for all angles $\alpha$ and $\Omega_\Psi=\{0\}$.
    
    In contrast, ${\mathcal R}_\Psi$ is the whole circle and so it is strictly larger than $\Omega_\Psi$.
    Indeed, notice that every point $\alpha_0$ of the circle converges under $\Psi$ to 0 both forward and backwards. 
    Hence, for every $\varepsilon>0$, there is a piece of a trajectory $x_0,\dots,x_n$ such that $x_0\in(0,\varepsilon/2)$, $x_i=\alpha_0$ for some $0\leq i\leq n$ and $x_n\in(2\pi-\varepsilon/2,2\pi)$. 
    Then, $x_i,x_{i+1},\dots,x_n,x_0,\dots,x_{i-1},x_i$ is a $\varepsilon$-chain from $\alpha_0$ to itself, \ie\ every point of the circle is chain-recurrent.
%\end{example}
\begin{figure}
 \centering
 \includegraphics[width=14cm]{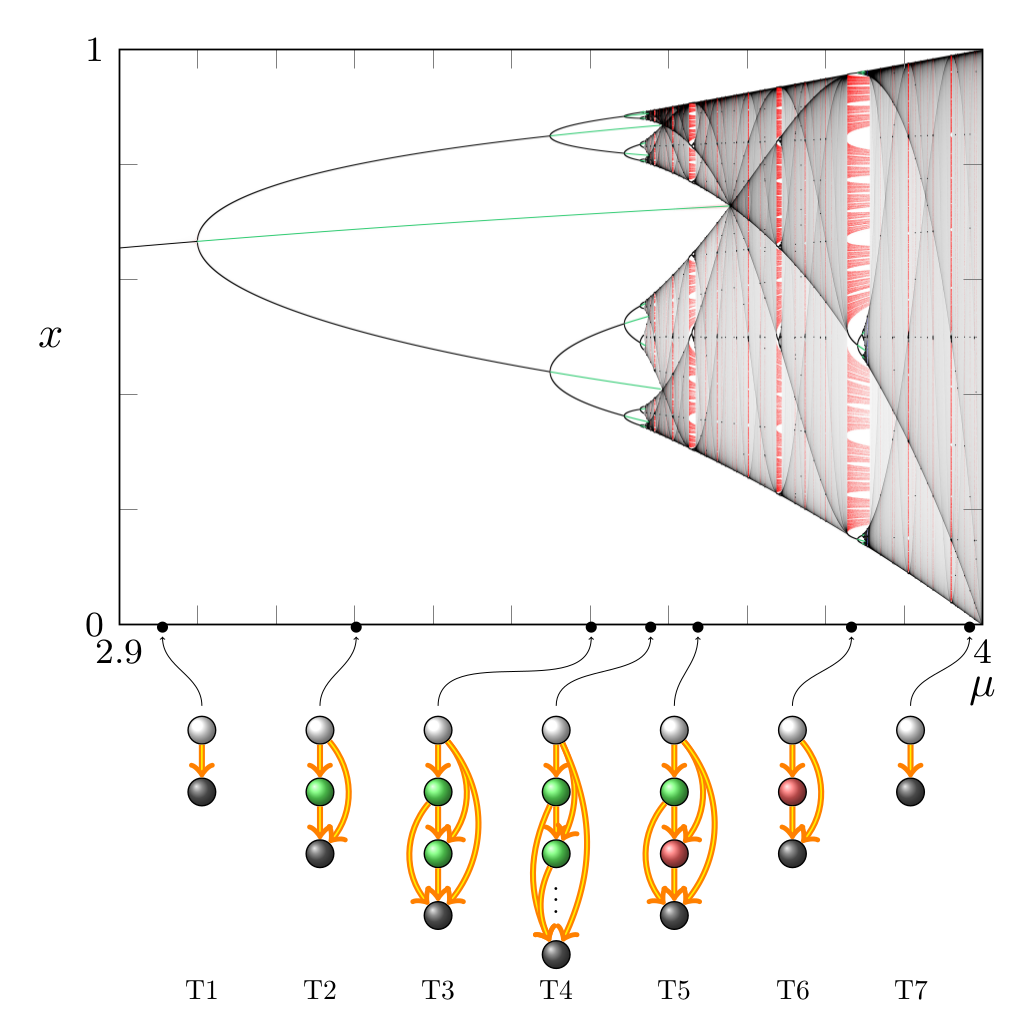}
  \caption{{\bf Bifurcations diagram and sample graphs of the logistic map.} 
  For each value of $\mu$, the attracting set is painted in shades of gray, depending on the density of the attractor, repelling periodic orbits in green, and repelling Cantor sets in red. 
  The ``black dots'' that are visible within the diagram are low-period periodic points. 
  They signal the presence of bifurcation cascades within some window. 
  The fact that some of them keep close to the line $x=c$ is a reflection of Singer's Theorem: when the attractor is a periodic orbit, $c$ belongs to its immediate basin. 
  For selected values of $\mu$ we also show, below the diagram, the graph of the corresponding logistic map.}
  \label{fig:full}
\end{figure}

%%%%%%%%%%%%%%%
\begin{figure}
 \centering
 \includegraphics[width=14.4cm]{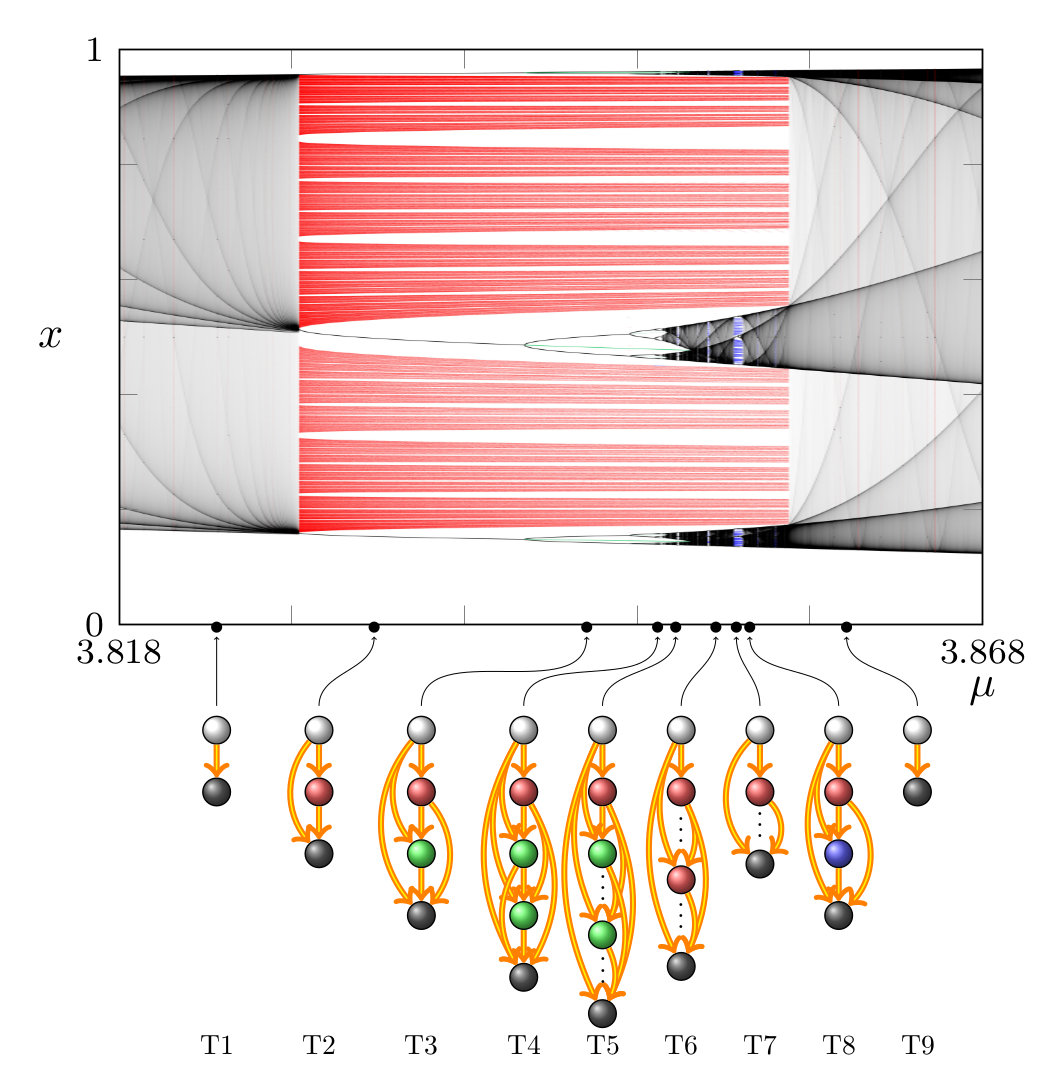}
 \caption{{\bf Towers of nodes in the period-3 window of the logistic map.}
%  The black dots on the $\mu$ axis are sample values for which we show some of the possible towers within t
    The period-3 window $W$ starts at $\mu\simeq3.8284$ and ends at $\mu\simeq3.8568$.
%  near the value where the period three orbit first appears, $\mu = 1+\sqrt 8\simeq3.8284$.
  At every $\mu$ within $W$, an invariant Cantor set node $N$ (in red in figure) arises, depending continuously on $\mu$. This is exactly the Cantor set discussed by Smale and Williams in~\cite{SW76} at $\mu=3.83$. 
  %  Red nodes denote fractal chaotic saddles. Blue and green nodes are periodic orbits. Black denotes attracting sets which occur at the bottom of the tower. 
  The widest white interval within the Cantor set is the $J_1(N)$ interval of the period-3 trapping region $\cT(N)$ running throughout $W$. 
  The periodic endpoint $p_1(N)$ of $J_1(N)$ is the top endpoint, the bottom one is $q_1=1-p_1(N)$.
  The blue Cantor set node $N'$ visible in figure is the node of a regular trapping region $\cT(N')$ nested in $\cT(N)$.
  As in Fig.~\ref{fig:full}, below the diagram we show the graph of the logistic map for selected values of $\mu$.
%  In the bifurcation diagram, the chain recurrent sets have the same coloring as the nodes shown in the towers below. The white node is the fixed point $0$. Green nodes are repelling periodic orbits. Red nodes are the Cantor set nodes visible in the diagram. Black nodes are attracting nodes.
}
 \label{fig:p3}
\end{figure} 
%%%%%%%%%%%%%%%

% %%%%%%%%%%%%%%%
% \begin{figure}
%  \centering
%  \includegraphics[height=9.25cm]{cr2ICz}\includegraphics[height=9.25cm]{crICz}
%  \caption{{\bf Examples of regular and flip interval cycles.}
% }
%  \label{fig:ic}
% \end{figure} 

%%%%%%%%%%%%%%%
\begin{figure}
 \centering
 \includegraphics[width=14cm]{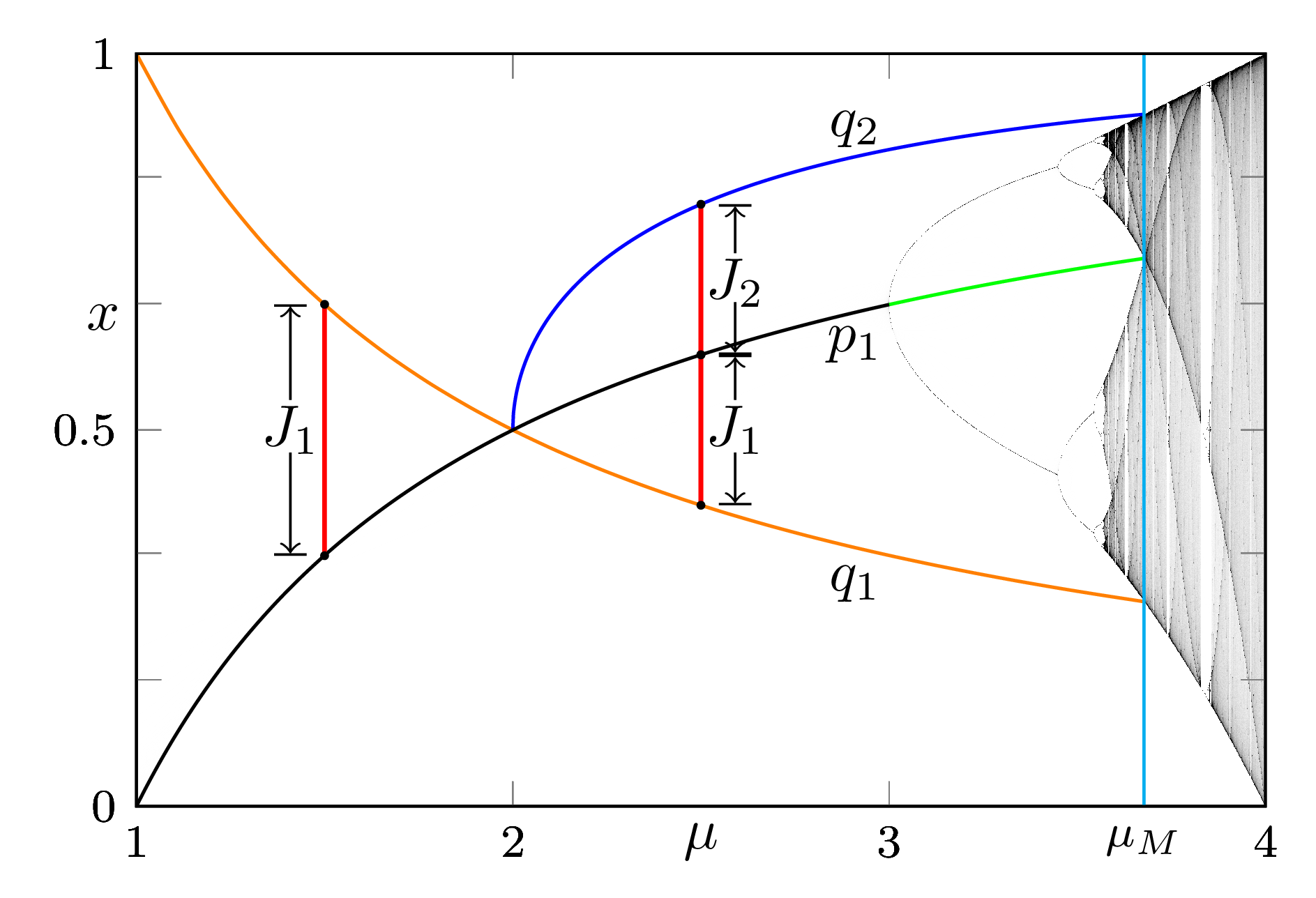}
 \caption{{\bf Flip and regular trapping regions associated to a periodic orbit node.} 
 The fixed point $p_1$ is a node for $1<\mu<\mu_M$ (see Eq.~\ref{Merge}). 
 Each node has its own $p_1$, $q_1$ and $J_1$.
 It is attracting for $\mu<3$.
% For $1<\mu<\mu_M$ (see Eq.~\ref{Merge}), the attracting fixed point, labeled $p_1$, is a node $N$. 
 For $\mu<2$, the trapping region associated with it consists of a single interval $\J_1=[p_1,q_1]$, where $q_1=1-p_1$. 
 As $\mu$ increases past the super-stable value $\mu=2$, the trapping region becomes flip and consists of two intervals $\J_1=[q_1,p_1]$ and $\J_2=[p_1,q_2]$, with $\ell_\mu(q_2)=q_1$. The flip trapping region ends when $p_1$ hits the chaotic attractor at $\mu_M$
 (Eq.~\ref{Merge}).
  %=1.5[ 1 + (19-3\sqrt{33})^{1/3} + (19+3\sqrt{33})^{1/3} ]\simeq3.67857$,
  The region $3.4\leq\mu\leq3.6$ is shown in greater detail in Fig.~\ref{fig:ic}.
}
 \label{fig:ic2}
\end{figure} 

\begin{figure}
 \centering
 \includegraphics[width=14.4cm]{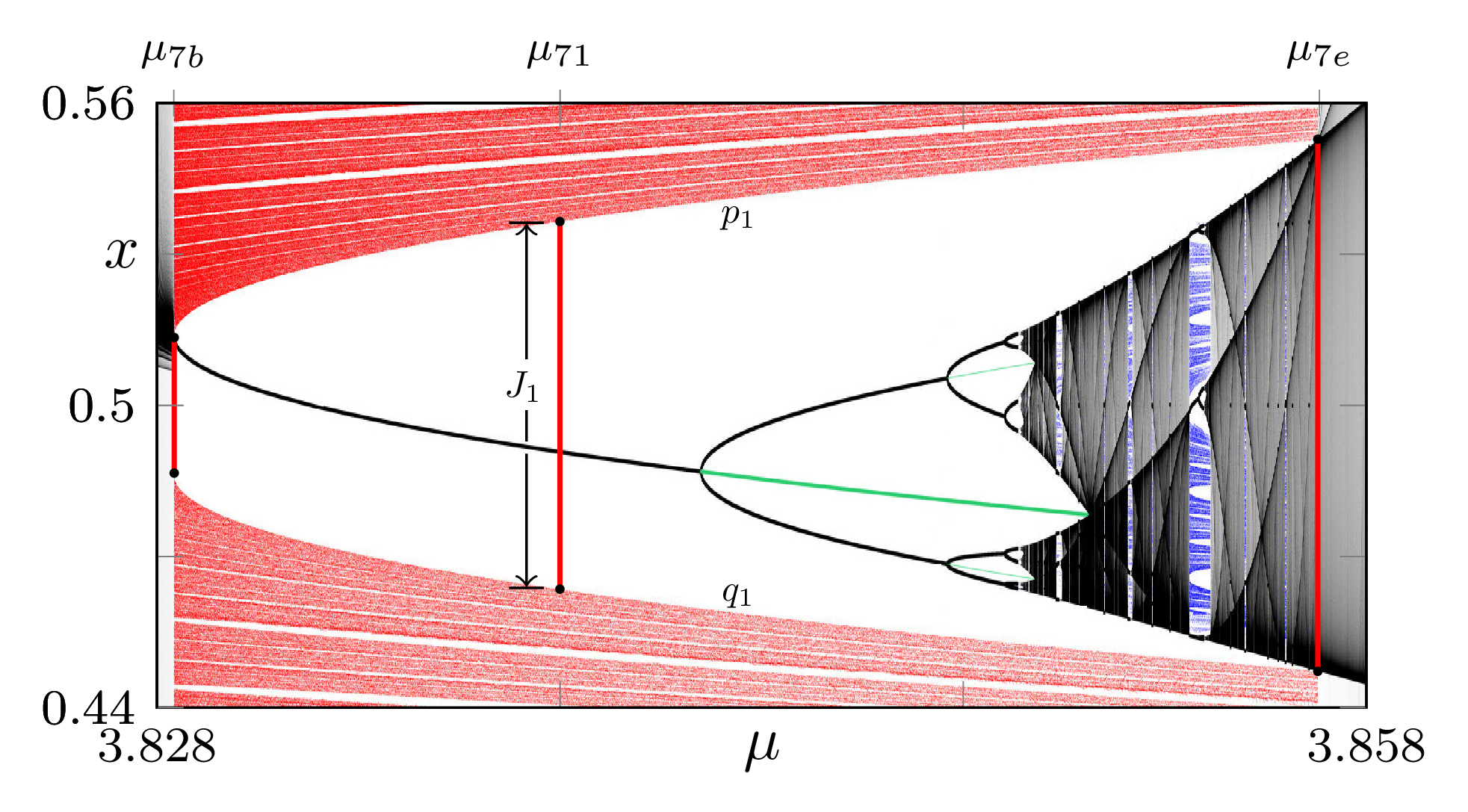}
 \caption{{\bf A trapping region associated to a Cantor set node.} 
    This window is a blowup of a region in Fig.~\ref{fig:p3}. 
    The red region is a Cantor set (for each $\mu$ in the window). 
    The Cantor set is a node. 
    The interval $J_1=[q_1,p_1]$ is a trapping region of $\ell_\mu^3$ for each $\mu$ in the window.
    Each node has a trapping region for $\ell_\mu$ and $J_1$ is the piece of the trapping region that contains the critical point. 
    %of associated to the node consisting in the red Cantor set.
    For each $\mu$ in the interior of the window, the point $p_1$ belongs to a repelling period-3 orbit within the red Cantor set and $\ell_\mu(q_1)=p_1$. 
    There is a $\mu$ 
    Within $J_1$ arises a bifurcation diagram qualitatively identical to the full one. 
    The Cantor set node in blue within the period-3 window of the diagram inside $J_1$ is the analog of the red Cantor set within the main period-3 window. 
    Fig.~\ref{fig:p3ic} shows the same region but with more detail.
%    for each $\mu$, the invariant Cantor set  the red points are the Cantor set node of all chain-recurrent points in the window that do not fall in the denoted by $p_1$ is a period-3 repellor. It separates the red Cantor set from the white basin of the attractor. The curve of points $q_1=1-p_1$ separates the Cantor set from the basin. 
}
 \label{fig:p3icb}
\end{figure} 
{%\bf Trajectories.}
%Here we restrict attention to discrete time dynamical systems. 
%
%We will need the following important result (e.g. see~\cite{Rob77,Rob98}) when we will classify the nodes of the logistic map.

\medskip
{\bf Definitions of trajectories.} For a map $\Phi$, we will say that the bi-infinite sequence $p_n$, $n\in\bZ$, is a {\bf trajectory} if $p_{n+1}= \Phi(p_n)$ for all $n\in \bZ$. 
Its forward (resp. backward) limit set {\BF$\omega(t)$} (resp. {\BF$\alpha(t)$}) is the set of the accumulation points of the trajectory for $n\to\infty$ (resp. $n\to-\infty$).
%\end{definition}
%
For some maps, the inverse is not unique. For the map $z\mapsto z^2$, each point other than $0$ has two inverses. 
Hence there will be infinitely many trajectories through a given $p_0 \ne 0.$ Two different trajectories through $p_0$ will have the same forward limit set but might have different backward limit sets.

%{\bf Edges.}
% \begin{definition}
%     Given a trajectory $t$, its forward (resp. backward) limit set {\BF$\omega(t)$} (resp. {\BF$\alpha(t)$}) is the set of the accumulation points of the trajectory for $n\to\infty$ (resp. $n\to-\infty$). 
% %    Similarly, its backward limit set $\alpha(t)$ is the set of its limit points for $n\to-\infty$.
% \end{definition}
%
\begin{proposition}
    For every trajectory $t$, either the trajectory lies entirely in a node or there are two distinct nodes $N_1,N_2$ such that $\alpha(t)\subset N_1$ and $\omega(t)\subset N_2$.
\end{proposition}
%
% \begin{definition}
%     There is an {\bf edge} from $N_1$ to $N_2\neq N_1$ if there exist a trajectory $t$ with $\alpha(t)\subset N_1$ and $\omega(t)\subset N_2$.
% \end{definition}
%If that node $\Omega$ is a compact set, then the distance of $\Phi^t(x)$ from $\Omega$ goes to $0$ as $t\to\infty.$
%
%\begin{definition}
{\bf Definition of attractor.}  Assume $X$ is a measure space. Following Milnor~\cite{Mil85}, we say that a closed invariant set $A$ is an {\bf attractor} if it satisfies the following conditions:
    \begin{enumerate}
        \item the basin of attraction of $A$, namely the set of all $x\in X$ such that $\omega(x)\subset A$, has strictly positive measure;
        \item there is no strictly smaller invariant closed subset $A'\subset A$ whose basin differs from the basin of $A$ by just a zero-measure set.
    \end{enumerate}
%    its basin of attraction, \ie\ the set of all points of $X$ asymptotic to $A$ under $\Phi$, has positive measure.
    We call a node an {\bf attracting node} if it contains an attractor, otherwise we call it a {\bf repelling node}. 
%\end{definition}

\bigskip
%\begin{remark}
    This definition of attractor is more appropriate for the logistic map than the more common definition that an attractor must attractor all points in some neighborhood. 
    
    For instance, The logistic map has a countable number of parameter values for which there is an attractor-repellor bifurcation at which a pair of periodic orbits is created. 
    At such points, the periodic orbit is attracting from one side and repelling from the other side. 
    Hence it is an attractor by our choice of definition. 
    Notice, moreover, that such an attractor-repellor orbit is a subset of an invariant Cantor set. We will show that this Cantor set is a node. Hence, by Milnor's definition, the node is an attracting node even though only part of it attracts.
    
    Similarly, there are parameter values where there are windows within windows infinitely deep, yielding a node which attracts almost every trajectory but whose basin does not contain an open set. Nonetheless, it is an attractor in the Milnor sense.
%\end{remark}
%
%We call a node an {\bf attractor}, also sometimes called a {\bf Milnor attractor}, if its basin of attraction, \ie, the set of points asymptotic to it under $\Phi$, has positive measure~\cite{Mil85}. 
%A non-trivial example of a Milnor attractor occurs at the Feigenbaum parameter value. In this case, the basin has empty interior and full measure. Moreover, the attractor has no periodic orbits but each neighborhood of it does. At a saddle-node bifurcation, there is a periodic orbit which is attracting form one side and repelling from the other. Using Milnor's criterion, we call it an attractor.
%We call a node a {\bf repellor} if it is not an attractor.

% %
% \tikzset{
%  LabelStyle/.style = { rectangle, rounded corners, draw, minimum width = 2em, fill = yellow!50, text = red, font = \bfseries },
% % VertexStyle/.append style = { inner sep=5pt, minimum size = 5pt},
% % VertexStyle/.append style = { inner sep=1pt, minimum size = 25pt},
% % VertexStyle/.append style = { inner sep=3pt, minimum size = 5pt},
%  EdgeStyle/.append style = {->, bend left=45,double=yellow}
% }
% %
% \tikzset{
%  LabelStyle/.style = { rectangle, rounded corners, draw, minimum width = 2em, fill = yellow!50, text = red, font = \bfseries },
% % VertexStyle/.append style = { inner sep=5pt, minimum size = 5pt},
% % VertexStyle/.append style = { inner sep=1pt, minimum size = 25pt},
% % VertexStyle/.append style = { inner sep=3pt, minimum size = 5pt},
%  EdgeStyle/.append style = {->, bend left=45,double=yellow}
% }
%          
\begin{figure}
 \centering
 \includegraphics[width=5.5cm]{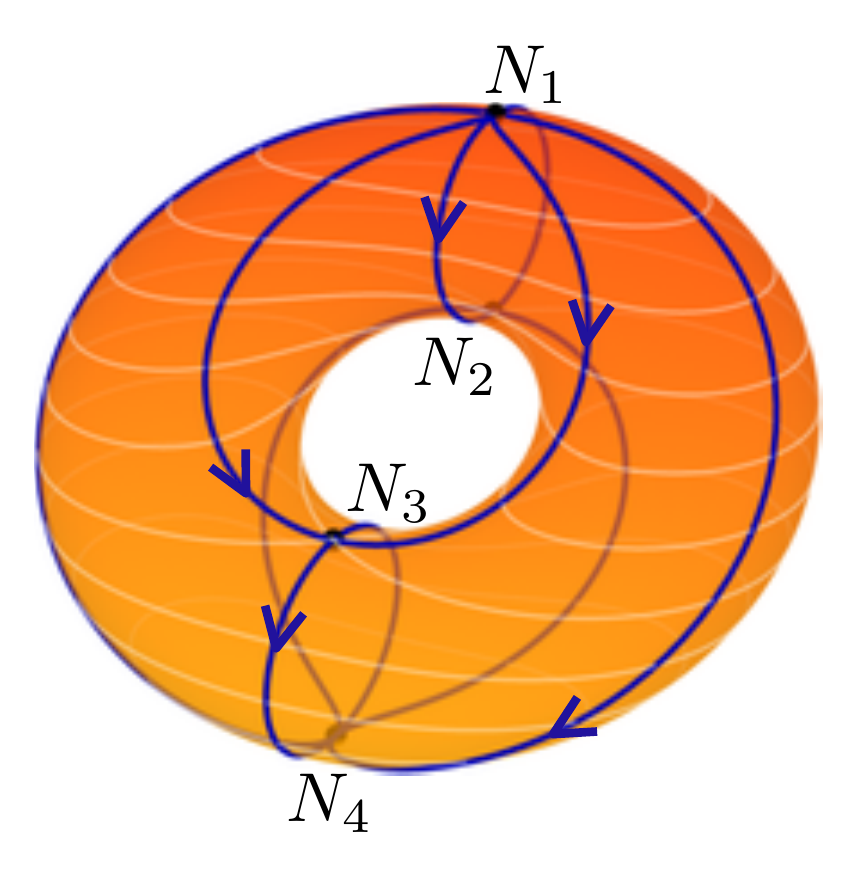}\hskip0.cm 
 \begin{tikzpicture}[scale=.3]
 \SetGraphUnit{5}
 \Vertex[L=$\bigO_1$]{A}
% \SO[L=$\bigO_2$,Lpos=180](A){B}
% \SO[L=$\bigO_3$](B){C}
% \SO[L=$\bigO_4$,Lpos=180](C){D}
% \SO[L=$\bigO_2$](A){B}
% \SO[L=$\bigO_3$](B){C}
% \SO[L=$\bigO_4$](C){D}
 \SO[L=$\bigO_2$](A){B}
 \SO[L=$\bigO_3$](B){C}
 \SO[L=$\bigO_4$](C){D}
 \tikzset{EdgeStyle/.append style = {bend left=45}}
 \Edge(A)(B)
 \Edge(A)(C)
 \Edge(A)(D)
 \tikzset{EdgeStyle/.append style = {bend right=45}}
 \Edge(B)(C)
 \Edge(B)(D)
 \tikzset{EdgeStyle/.append style = {bend right=0}}
 \Edge(C)(D)
\end{tikzpicture}
\hskip0.cm
\tikzset{
 LabelStyle/.style = { rectangle, rounded corners, draw, minimum width = 2em, text = black, font = \bfseries },
 EdgeStyle/.append style = {->, bend right=62,double}
}
\scalebox{1}{
\begin{tikzpicture}[scale=.3]
 \SetGraphUnit{5}
 \Vertex[L=$\bigO_1$]{A}
 \Vertex[L=$\bigO_{m}$,a=-90,d=5.8cm]{B}
 \Vertex[L=$\bigO_{n}$,a=-90,d=11.6cm]{C}
% \SO[L=$\bigO_m$](B){C}
% \SO[L=$\bigO_{n}$](B){C}
 \Vertex[L=$\bigO_\infty$,a=-90,d=17.4cm]{D}
 \Edge[local=true,label=$\bigd$,color=white,labelstyle={fill=white},style={bend left=0}](A)(B)
 \Edge[local=true,label=$\bigd$,color=white,labelstyle={fill=white},style={bend left=0}](B)(C)
 \Edge[local=true,label=$\bigd$,color=white,labelstyle={fill=white},style={bend left=0}](C)(D)
 \Edge(A)(B)
 \Edge(A)(C)
 \Edge(A)(D)
 \tikzset{EdgeStyle/.append style = {bend left=45}}
 %\Edge[local=true,label=$\bigd$,color=white,labelstyle={fill=white},style={bend left=0}](B)(C)
 \Edge(B)(C)
 \Edge(B)(D)
 % \tikzset{EdgeStyle/.append style = {bend right=50}}
 \Edge(C)(D)
\end{tikzpicture}
}
 \caption{{\bf An example of \gr graph.} {\bf (LEFT)} Dynamics induced on the 2-torus by the gradient vector field of the height function. In this case the Lyapunov function is the height function itself, some level set of which is shaded in white. In blue are shown the heteroclinic trajectories joining the critical point (which are exactly the invariant sets of this dynamical system). {\bf (CENTER)} The \gr graph of the dynamical system on the left. In this case it is a 4-levels tower. (RIGHT) An infinite tower.
 }
 \label{fig:pp}
\end{figure}
{\bf The \gr graph.} 
Conley~\cite{Con72,Con78} realized that chain-recurrence could be used to define a graph of a dynamical system.
His investigations concerned dynamical systems that come from ordinary differential equations on compact spaces but,
over the years, his results have been extended to several other settings; in particular: continuous maps~\cite{Nor95b}, semi-flows~\cite{Ryb87,HSZ01,Pat07}, non-compact~\cite{Hur91,Pat07} and even infinite-dimensional spaces~\cite{Ryb87,MP88,CP95,HMW06} (notice that,  
throughout this article, we sort multiple citations in the order of their year of publication).
%
% \allred
% \begin{definition}
%      We say that $x\in X$ is {\bf gradient-like} for $\Phi:X\to X$ if there is a trajectory $t$ passing through it such that $\alpha(t)$ and $\omega(t)$ are contained in different nodes.
% \end{definition}
\allblack

%\begin{definition}
{\bf Definition of Liapunov function.}
   A {\bf Lyapunov function}~\cite{WY73,Nor95b} for $\Phi:X\to X$ is a continuous function $L$ such that:
   \begin{enumerate}
       \item $L$ is constant on each node;
       \item $L$ assumes different values on different nodes;
       \item $L(\Phi(x))< L(x)$ if and only if  $x$ is not chain-recurrent.
%       \in X\setminus{\mathcal R}_\Phi$.
   \end{enumerate}
   %
%\end{definition}
%

%
\begin{theoremN}[Norton, 1995~\cite{Nor95}]
    Let $\Phi$ be a dynamical system on a compact metric space $X$. Then there is a Lyapunov function for $\Phi$. 
    % \allred
    % In other words, every point $x\in X$ either belongs to ${\mathcal R}_\Phi$ or is gradient-like for $\Phi$.
    \allblack
\end{theoremN}
%
%We do not explicitly use the Conley theorem but we believe it gives a conceptual picture that may be helpful to the reader. 
%In particular, each point $x\in X$ either belongs to a node of ${\mathcal R}_\Phi$ or every trajectory $t$ passing through it 
The Discrete Conley Theorem allows to associate a graph to any dynamical system as follows.

% {\bf Graphs and Lyapunov functions}.
% \allblack
% The dynamics outside of the nodes is always {\bf gradient-like}, namely there is a continuous function $L:X\to\bR$ such that:
% %{\allred We want L to decrease, not increase so I changed the sign in 1. I hope that is OK. And is this result below just for invertible dynamics? Should it say $<$:}
% \begin{enumerate}
%   \item $L$ is constant on each node; 
%   \item $L$ assumes different values on different nodes;
%   \item $L(\Phi^tx)< L(x)$ for all $t>0$ and when $x$ not in a node~\cite{MM02}.
% \end{enumerate}
% %
% In particular, if there is an edge from node $N_1$ to node $N_2$, then there cannot be an edge from $N_2$ to $N_1$. 
%
%\begin{definition}

\bigskip\noindent {\bf Definition of graph of a dynamical system.}
    The {\bf \gr graph} $\Gamma$ of a dynamical system $\Phi:X\to X$ is a directed graph whose nodes are the nodes of $\Phi$.
%    There is an {\bf edge} from $N_1$ to $N_2\neq N_1$ if there exist a trajectory $t$ with $\alpha(t)\subset N_1$ and $\omega(t)\subset N_2$.
   $\Gamma$ has an edge from node $N$ to node $N'$ if and only if there exist a trajectory $t$ of $\Phi$ with $\alpha(t)\subset N$ and $\omega(t)\subset N'$.
%   $N_1$ is upstream from $N_2$ and $N_2\neq N_1$.
%\end{definition}

\bigskip
\noindent{\bf The graph of a dynamical system has no loops.} Notice that from the Discrete Conley Theorem it follows that, if there is an edge from $N$ to $N'$, there cannot be an edge from $N'$ to $N$. 
Moreover, there cannot be any loops, that is, there cannot be a collection of nodes $N_1,\dots,N_k$ such that there is an edge from $N_i$ to $N_{i+1}$, for $i=1,\dots,k-1$, and from $N_k$ to $N_1$.

\tikzset{
 LabelStyle/.style = { rectangle, rounded corners, draw, minimum width = 2em, fill = yellow!50, text = red, font = \bfseries },
 EdgeStyle/.append style = {->, bend left=45,double=yellow}
}
\bigskip\noindent {\bf Definition of tower.}
   We say that a \gr graph $\Gamma$ is a {\bf tower} if there is an edge between every pair of distinct nodes of $\Gamma$.
\bigskip   
   
%\end{definition}
%
Towers are the kind of \gr graph that this article is about.
We show in Fig.~\ref{fig:pp} some example of tower with finitely and infinitely many nodes. Since we assume $X$ to be compact, in a tower there is always a lowest node and it contains a Milnor attractor. An elementary example of dynamics with such a graph is the gradient flow of a Morse function on the 2-torus (see Fig.~\ref{fig:pp}, right). 

\section{The logistic map}
\label{sec:lm}
The logistic map
\beq
\label{eq:lm}
\ell_\mu(x) = \mu x(1-x),\;x\in[0,1],\;\mu\in[0,4],
\eeq
is among the simplest continuous maps giving rise to a non-trivial dynamics. 
We recall that a continuous map $f:[0,1]\to[0,1]$ with $f(0)=f(1)=0$ and for which there is a point $c\in(0,1)$ such that $f$ is strictly increasing (resp. decreasing) for $x<c$ and strictly decreasing (resp. increasing) for $x>c$, is called {\bf unimodal}. Moreover, a unimodal map $f$ is {\bf S-unimodal} if it is at least $C^3$ and its Schwarzian derivative (see~\cite{Sin78})
%$$
%    S(x)=\frac{f'''(x)}{f'(x)}-\frac{3}{2}\left[\frac{f''(x)}{f'(x)}\right]^2
%$$
is negative for every $x\neq c$. Notice that the logistic map is a S-unimodal map.

\medskip

For $\mu\in(0,1]$, the point $0$ is the unique attractor of $\ell_\mu$ and its basin of attraction is the whole segment $[0,1]$. From now on, we will assume that $\mu\in(1,4]$.

%\allred
The focus of the present work is on the edges in the \gr graph of the logistic map. Indeed, as we will discuss below, while the structure of the invariant attracting and repelling sets have been thoroughly studied at several levels of generality, no one seems to have focused on the edges of the \gr graph for the logistic map. Hence we do.

We have written this section so that the reader will gain a detailed picture of the possible chain-recurrent sets and the connections between them. To do this, we have a number of propositions whose proofs are often quite simple. The propositions are there to create a mental picture of the graphs. 

The literature we refer to discusses non-wandering sets, whereas we investigate chain-recurrent sets. 
The biggest difference between these two approaches occurs at the final parameter value of each window. 
At those values, there is a chaotic attractor, consisting of intervals, and a repelling Cantor set, and these two non-wandering sets have a common periodic orbit. 
There is a single node, consisting in a finite union of intervals, that includes both non-wandering sets and the gaps in the Cantor set. 
%This node is a union of intervals. 

Notice, finally, that in all our proofs below we never use any property specific to the logistic map but rather those that come from it being a S-unimodal map. Hence, all our results actually hold in general for any S-unimodal map.
\allblack
\subsection{Invariant sets of the logistic map}
The classification of attractors and repellors of S-unimodal and unimodal maps was an important achievement of 1-dimensional dynamics. Below we recall these results, that we state for the specific case of the logistic map, since they are the starting point of our work and we are going refer to them often in the remainder of this section.

\bigskip\noindent{\bf Attractors.}
%
%The next definition is crucial for the study of the dynamics of the logistic map and will be used throughout the remainder of the section.
The first fundamental result found about attractors in S-unimodal maps is the uniqueness of periodic attractors.
\begin{definition}
    The {\bf immediate basin} of an attracting periodic orbit is the union of all connected components of the basin of the orbit that contain a point of it. 
    %In particular, it is an open interval.
\end{definition}
\begin{theoremS}[Singer, 1978~\cite{Sin78}]
    If $\ell_\mu$ has an attracting periodic orbit $P$, then it has no other attracting periodic orbit and the critical point $c$ belongs to the immediate basin of $P$. 
\end{theoremS}
The classification of attractors and their uniqueness was proved by Guckenheimer in case of S-unimodal maps and by Jonker and Rand in case of unimodal maps (see also Thm.~4.1 in~\cite{dMvS89}). 
\begin{definition}
    \label{def:trap}
    By a {\bf trapping region} we mean a collection $\cT$ of $k$ intervals with disjoint interiors $J_1,\ell_\mu(J_1),\dots,\ell_\mu^{k-1}(J_1)$ 
    such that:
    \begin{enumerate}
    \item $c\in\text{int}(\J_1)$;
    \item $\ell_\mu^k(J_1)\subset J_1$.
    \end{enumerate}
    We denote by $J_i$ the sets $\ell_\mu^{i-1}(J_1)$, $1\leq i\leq k$, and by
    $\Jall=\Jall(\cT)$ the union of the interiors of the $\J_i$.
    
%    By a {\bf chaotic attractor} we mean an attracting trapping region with a dense trajectory.
\end{definition}
It is unusual to ask that $c$ belongs to the trapping region but this restriction is automatically satisfied for the trapping regions we are interested in.

\begin{theoremA}[Guckenheimer~\cite{Guc79}; Jonker and Rand~\cite{JR80}]
    %For each $\mu\in(0,4]$, t
    The map $\ell_\mu$ has exactly one attractor and this attractor is one of the following types:
    \begin{enumerate}
        \item {\bf Periodic}; this attractor is a periodic orbit. This case includes when the orbit attracts only from one side;
        \item {\bf Chaotic}; this attractor is a trapping region with a dense trajectory. 
%        [Authors' note: such attractors have a dense trajectory].
        \item {\bf Almost Periodic}~\cite{Mil85}, also sometimes called ``odometer'' or ``solenoid''; this attractor is a Cantor set on which $\ell_\mu$ acts as an adding machine (e.g. see Chap.~2 in~\cite{Bue12}).
    \end{enumerate}
    The basin of attraction has full measure in all three cases but only in the first two does the basin have non-empty interior. The critical point $c$ is always in the basin.
\end{theoremA}
Bifurcation plots showing the dependence of the attractor on the parameter $\mu$ for values to the left of the so-called {\bf Myrberg-Feigenbaum } parameter value
%or {\bf Feigenbaum  parameter value} 
$\mu_{FM}\simeq3.5699$~\cite{Fei78} appeared in several publications in the $1970$'s but, to the best of our knowledge, the first picture of the full bifurcation diagram %(Fig.~\ref{fig:lmbd}) 
appeared first in an article by Grebogi, Ott and Yorke in 1982~\cite{GOY82}.

The almost periodic case occurs for those parameter values for which the graph has infinitely many nodes. For each point $x_0$ in the Cantor set attractor and each $\varepsilon>0$, there is a periodic point $x_{\varepsilon}$ such that the $n$-th iterate of the map on $x_0$ and the $n$-th iterate of the map on $x_{\varepsilon}$ stay within $\varepsilon$ of each other for all time $n\geq0$. 
At $\mu_{FM}$, every period orbit has period $2^k$ for some $k$ and none of them belongs to the Cantor set. They converge to the Cantor set attractor as $k\to\infty$.

We write the parameter space as $(1,4]=\AP\cup\AC\cup\ALP$, where the union is disjoint, ${\mathcal A}_P$ is the set of parameters for which the attractor is a periodic orbit that is not a one-sided attractor, $\ALP$ the set of those for which it is a Cantor set and ${\mathcal A}_C$ is the set of all other parameters, which includes all those for which the attractor is chaotic.
The set ${\mathcal A}_P$ is open, a trivial consequence of the stability of periodic orbits under small perturbations, and dense, as proved independently by Lyubich~\cite{Lyu97} and Graczyk and Swiatek~\cite{GS97}.
The complement of $\AP$ is a Cantor subset of $[\mu_{FM},4]$. Notice that $\mu_{FM}\in\ALP$ and $4\in\AC$. Jakobson~\cite{Jak81} proved in 1981 that $\AC$ has positive measure. It was proved in 2002 by Lyubich~\cite{Lyu02} that $\ALP$ has measure zero.

\bigskip\noindent{\bf Repellors.}
We come now to the results about the decomposition of the whole non-wandering set of $\ell_\mu$,
which we denote by $\Omega_{\ell_\mu}$. 
This result is a generalization of the following well-known result of Smale (see Sec.~I.6 in~\cite{Sma67} for details) called ``Spectral Decomposition for Diffeomorphisms''. 
In this theorem, Smale showed that, for a smooth manifold $M$ and a Axiom-A diffeomorphism $f$ on $M$, there is a unique way to write non-wandering set $\Omega_f$ as the finite union of elementary pairwise disjoint components $\Omega_j$, on each of which the map has a dense orbit.
\begin{definition}
    We say that a set $P$ equal to the finite union of $k$ isolated periodic orbits is a {\bf cascade segment} if the periodic orbits in $P$ belong to the same cascade, whose first orbit has period $n$, and their period are equal to $n,2n,\dots,2^{k-1}n$ for some $k$.
\end{definition}
\begin{theoremW}[van Strien, 1981~\cite{vS81}]
%    For each $\mu\in(1,4]$, 
     Each $\ell_\mu$ has $p+1$ trapping regions $\cT_j$, where $p$ can be infinite, such that the following properties hold:
    \begin{enumerate}
    \item $\ell_\mu(\text{cl}(\Jall(\cT_j)))\subset \Jall(\cT_j)$ for all $0\leq j<p$.
    \item  The $\cT_j$ are nested: $\text{cl}(\Jall(\cT_j))\subset\text{cl}(\Jall(\cT_{j-1}))$ for all finite $1\leq j\leq p$.
    \item Set $K_j=\text{cl}(\Jall(\cT_{j-1}))\setminus\text{cl}(\Jall(\cT_{j}))$. Then
    $\Omega_{\ell_\mu}$ is the union of the following $p+1$ closed forward invariant sets:
    $$
    \Omega_j=\bigcap_{n\geq0}\ell^n_\mu\left(K_j\right), 0\leq j<p,
    $$
    $$
    \Omega_p = \Omega_{\ell_\mu}\cap\left[\bigcap_{j=0}^p\text{cl}(\Jall(\cT_{j}))\right].
    $$
    \item $\Omega_0=\{0\}$;
    \item Each $\Omega_j$, $0<j<p$, is the union of a Cantor set $C_j$ and a cascade segment. %finite number of isolated periodic orbits;
    The action of $\ell_\mu$ on $C_j$ is a subshift of finite type with a dense orbit. 
%    The periodic orbits are a cascade segment.
    \item Each $\Omega_j$ is hyperbolically repelling for $j<p$;
    \item $\Omega_p$ is the unique attractor of $\ell_\mu$ and it is not, in general, hyperbolic [Authors' note: the attractor fails to be hyperbolic at the beginning and end of each window, see~Prop.~\ref{thm:attracting})].
    \item $\Omega_i\cap\Omega_j=\emptyset$ for $0\leq i,j<p$ with $i\neq j$.
    \item When $p<\infty$, $\Omega_{p-1}\cap\Omega_p$ is empty except when $\Omega_p$ is not hyperbolic, in which case it contains a single periodic orbit.
    \item When $p=\infty$, all $\Omega_j$ are disjoint, $\Omega_\infty$ is a Cantor set and the action of $\ell_\mu$ on it is an adding machine.
    \end{enumerate}
    
\end{theoremW}
The first version of this theorem was given by Jonker and Rand~\cite{JR80} in case of unimodal maps. The specialization, used in the statement above, to S-unimodal maps is due to van Strien~\cite{vS81}. Several other versions and generalizations of this theorem are available in literature, \eg\ Holmes and Whitley~\cite{HW84}, Blokh and Lyubich~\cite{BL91}, Blokh~\cite{Blo95}, Sharkovsky et al.~\cite{SKSF97}. Possibly the most thorough version is Thm.~4.2 in~\cite{dMvS93} by van Strien and de Melo.

To our knowledge, no bifurcation diagram showing a repelling Cantor set has appeared to date in literature. 
Our Fig.~\ref{fig:full} and~\ref{fig:p3} illustrate the content of the theorem above by showing the attractors (in shades of grey) together with some repelling periodic orbits (in green) and Cantor sets (in red and blue).

The Non-Wandering Theorem above has some shortcoming for our purposes, e.g.: 1) the problem of what determines the number of cyclic trapping regions of $\ell_\mu$ is not addressed; 2) unlike the nodes in our approach, the decomposition sets $\Omega_i$ are not necessarily pairwise disjoint; 3) nothing is explicitly stated about the dynamics of points outside the basin of attraction of the attractor; 4) the decomposition sets $\Omega_j$ contain, in general, more than one node. In particular, the restriction of the logistic map to them cannot a dense orbit, 
%the restriction of the map on the repelling decomposition sets $\Omega_i$ is not transitive in general, 
how happens instead in case of Axiom-A diffeomorphisms.
In the remainder of this section we show that these shortcomings naturally disappear by replacing $\Omega_{\ell_\mu}$ with ${\mathcal R}_{\ell_\mu}$.

\bigskip\noindent
{\bf Homtervals.} 
We recall a last theorem that we are going to use several times in this section.
\begin{definition}
   A closed interval $J\subset [0,1]$ is a {\bf homterval} for $\ell_\mu$ if $c$ is not in the interior of $\ell^k_\mu(J)$ for any integer $k\geq0$.
%  A closed interval $J\subset [0,1]$ is a {\bf homterval} for $\ell_\mu$ if, for each integer $k\geq0$, $c$ is not in the interior of $\ell^k_\mu(J)$.
\end{definition}
Notice that, equivalently, $J$ is a homterval for $\ell_\mu$ if all of its iterates $\ell^k_\mu$, $k=0,1,\dots$, are strictly monotonic on $J$.
\begin{theoremH}[Guckenheimer, 1979~\cite{Guc79}]
    \label{thm:Guc79}
    Assume that $\ell_\mu$ admits a homterval $J$. Then the attractor of $\ell_\mu$ is periodic and $\omega(x)$ is a periodic orbit for each $x\in J$. 
%    Let $J\subset[0,1]$ be a closed interval such that, for each $k\geq1$, $c$ in not in the interior of $\ell_\mu^k(J)$. Then the attractor of $\ell_\mu$ is a periodic orbit and 
%    $J$ lies in its basin.
\end{theoremH}
Notice that, in particular, $\omega(x)$ must be equal to the attractor for almost all points of $J$.

%Remark: The next case concerns homtervals, so it might be helpful to the reader if we explain how an invariant homterval can occur. 
    
The following case is particularly important to us. When a period-$k$ orbit is attracting for some $\mu_0$, and then $\mu$ is increased so that the periodic orbit becomes unstable, it period-doubles at some $\mu_1$. 
Then, at some $\mu_2>\mu_1$, the new attracting period-$2k$ orbit is superstable. 
For each $\mu\in(\mu_1,\mu_2]$, each point $z$ of the now-repelling period-$k$ orbit lies in an interval $K_z$ whose endpoints are period-$2k$ points. 
$K_z$ is a homterval that is invariant under $\ell_\mu^{k}$. 
%If $\mu$ is sufficiently small (the $\mu$ has not been reached for which the period-$2k$ orbit is superstable) these intervals are invariant homtervals. 
In particular, all points in the homterval $K_z$, except $z$, are attracted to the period-$2k$ orbit. 
\subsection{Chain-recurrence and graph for the logistic map.} 
This subsection contains our contributions for this article.
We start by introducing cyclic trapping regions (Def.~\ref{def:ic}), a kind of trapping region having  
%Cyclic trapping regions always have 
a periodic orbit on its boundary. %(Prop.~\ref{thm:ic2node}). 
%Nodes, their 
Cyclic trapping regions
and their ``accessible'' periodic orbits (Def.~\ref{def:accessible}) 
come in two flavors: regular and flip (Def.~\ref{def:regflip}). 

{\bf Pairing windows and regular trapping regions.} For the logistic map, there is a 1-1 correspondence that pairs each repelling regular trapping region with a ``window'' in the bifurcation diagram (see Fig.~\ref{fig:full} for the full diagram and Fig.~\ref{fig:p3} for a detail of the period-3 window).
A window is a maximal interval $[\mu_0,\mu_1]$ that has a regular cyclic trapping region that persists throughout the interval.
The observed cascade of the window lies within the cyclic trapping region.
Every window comes with a node which is the invariant Cantor set of chain-recurrent points that do not fall into the trapping region. 
Figure~\ref{fig:p3} shows, for each value of $\mu$, the Cantor set of the period-3 window (in red) and the one of its period-9 subwindow (in blue). 
For each $\mu$, the largest white gap in the Cantor set is one of the intervals of the regular trapping region corresponding to that window, which we call $\J_1$. 
The number of such intervals, namely the period of the cyclic trapping region, coincides with the period of the window and with the period of the orbit at the boundary of the trapping region. 
This orbit is the unique ``accessible'' orbit of the Cantor set (see Def.~\ref{def:accessible} and Thm.~\ref{thm:accessible}).

A period-$2k$ flip trapping region is created as $\mu$ increases when the derivative at the period-$k$ point $\frac{d}{dx}\ell^k_\mu(p_1)$ becomes negative (see Fig.~\ref{fig:ic2}).

We show that each repelling node is either a Cantor set, in case the corresponding cyclic trapping region is regular, or a flip periodic orbit, in case is flip (Prop.~\ref{thm:repelling}). We give an analogue classification of attracting nodes (Prop.~\ref{thm:attracting}). 
\allblack

% s correspond to the periodic orbits in the bifurcation cascade within a window (Fig.~\ref{fig:ic}, \ref{fig:tobecreated}). Each such periodic orbit is at the boundary of a flip trapping region, whose period is double the period of the orbit. 
{
We use cyclic trapping regions to study the structure and properties of the nodes and edges of the \gr graph of ${\ell_\mu}$. 
%We show that there is a pairing between nodes with positive distance from the critical point and cyclic trapping regions: 
Here are some key structural results:
\begin{itemize}
    \item {\bf Each node has its own trapping region.} Given a node $N$, the minimum distance between $N$ and $c$ is achieved at a period-$k$ point $p_1$ in $N$ and the map $\ell^{2k}_\mu$ leaves invariant the interval $\J_1$ with endpoints $p_1$ and $1-p_1$ (Prop.~\ref{thm:node2ic}). 
    The interval $\J_1$ it is one of the intervals of a cyclic trapping region $\cT(N)$. 
    The period of the trapping region is $k$ if it is regular and $2k$ if it is flip. 
    It follows that no point of $N$ falls under the map into the interior of the trapping region. %$\J_1$.
    %this $\J_1$ is the first interval of a cyclic trapping region. 
    \item {\bf Each trapping region has its own node.} Given a cyclic trapping region $\cT$, its interval $\J_1$ contains the critical point. There is a unique node $\text{Node}(\cT)$ containing the periodic orbit (passing through $p_1$) at its boundary (Def.~\ref{def:N(T)}). 
    \item {\bf The trapping region and the node have one periodic orbit in common.} The point $p_1$ belongs to this periodic orbit.
    \item {\bf Given any two distinct nodes $N$ and $N'$, one must be in the trapping region of the other.} Specifically, if $p_1(N)$ is closer to $c$ than $p_1(N')$, then $N$ is in the trapping region of $N'$.
    \item {\bf The graph is a tower.} For each chain-recurrent point $x\in\Jall(N)$, there is a trajectory $t$ passing through $x$ that converges backwards to $p_1(N)$ (Prop.~\ref{prop:upstream}). This means that there is an edge between each pair of nodes, namely the graph is a tower (Thm~\ref{thm:tower}). This tower is infinite if and only if $\mu\in\ALP$. In  Figs.~\ref{fig:full} and~\ref{fig:p3} we show several examples of towers in the logistic map.
\end{itemize}
%In both cases, $\J_1$ is the first interval of a cyclic trapping region. 
%One of the endpoints   of which lie the periodic orbit passing through $p_1$ (Cor.~\ref{cor:pairing}). 
%The period of the trapping region is $k$ if it is regular and $2k$ if it is flip. 

%We have, therefore, a cyclic trapping region for each repelling node. 

%Finally,  %(Thm.~\ref{thm:spectradec}).

%In ??? we prove that, for every node $N$ with positive distance from the critical point, there exist a cyclic trapping region 
%We show ultimately that the graph is always a tower. This tower is infinite if and only if $\mu\in\ALP$.
}

\bigskip\noindent{\bf Windows.} We start by introducing and refining some fundamental concepts that will be in the background of all statements in this section.
% %
% \begin{definition}
%     \label{def:ic}
%     A {\bf period-\BF$k$ (cyclic) trapping region} is a trapping region $\cT=\{\J_1,\dots,\J_k\}$ such that the following are satisfied:
%     %
%     \begin{enumerate}
%         \item $c\in\interior(\J_1)$;
%         \item $\lmu(\J_i)=\J_{i+1}$, $i=2,\dots,k-1$; $\ell_\mu(\J_{k})=\J_1$;
%         \item $\lmu(\J_1)\subset \J_2$;
%         \item in (2) and (3), $\lmu$ maps endpoints to endpoints.
%     \end{enumerate}
%     %
%     % We denote by $\Jall=\Jall(\cT)$ the union of the interiors of the $\J_i$.
% \end{definition}
%
% \begin{definition}
%     \label{def:ic}
%     A {\bf period-\BF$k$ (cyclic) trapping region} is a trapping region $\cT=\{\J_1,\ell_\mu(J_1),\dots,\ell^{k-1}_\mu(\J_1)\}$ such that the following are satisfied:
%     %
%     \begin{enumerate}
%         \item there is a periodic point $p_1$ such that the endpoints of $J_1$ are $p_1$ and $1-p_1$;
%         \item $\lmu^{k}(\J_1)\subset\J_{1}$.
%     \end{enumerate}
%     %
%     We denote by $J_i$ the set $\ell_\mu^{i-1}(J_1)$, by $|\cT|$ the period of a cyclic trapping region $\cT$ and by $\orbit(\cT)$ the unique orbit at the boundary of $\cT$.
% \end{definition}
%
\begin{definition}
    \label{def:ic}
    A {\bf period-\BF$k$ (cyclic) trapping region} is a trapping region $\cT=\{\J_1,\ell_\mu(J_1),\dots,\ell^{k-1}_\mu(\J_1)\}$ such that $J_1$ has endpoints $p_1$ and $q_1=1-p_1$ and the one denoted by $p_1$ is periodic.
    We denote by $|\cT|$ the period of  $\cT$ and by $\orbit(\cT)$ the periodic orbit containing $p_1$.
\end{definition}
\begin{remark}
    The requirement that each interval of a period-$k$ trapping region be an iterate of $J_1$ can be relaxed by allowing each $J_i$ to be larger than $\ell_\mu(J_{i-1})$,
    as long as $\ell_\mu(J_k)\subset J_1$. For such generalized trapping region, therefore, we have that 
    $\ell_\mu(J_{i})\subset J_{i+1}$ for all $i$, where we identify $J_{k+1}$ with $J_1$.
    
    There are two limiting cases. One is the standard trapping region defined above, when all inclusions are equalities except for the last one: $\ell_\mu(J_k)\subset J_1$. This choice is minimal: if any $J_i$, $i\neq1$, is taken smaller, then there is no trapping region with that $J_1$ and that $J_i$. The other one is the one where all inclusions are equalities except for the first one: $\ell_\mu(J_1)\subset J_2$. This choice is maximal: if $J_2$ were chosen any larger, $\ell_\mu(J_k)$ would be larger than $J_1$. In a maximal trapping region $\ell_\mu$, restricted to any $J_i$, sends interiors into interiors and endpoints into endpoints.
    
    The results of this section do not depend on which particular definition of trapping region is used. In Prop.~\ref{thm:repelling} and in the pictures we use the ``maximal'' definition because it is more convenient.
\end{remark}
%Before going to the classification of attracting nodes, we need first to discuss in some length how trapping regions change with the parameter $\mu$.
As Figs.~\ref{fig:pric} and~\ref{fig:p24ic} suggest, there are two distinct kinds of cyclic trapping regions.
\begin{definition}
\label{def:regflip}
    Let $\cT$ be a cyclic trapping region of $\ell_\mu$
    such that $\orbit(\cT)$ is a period-$k$ orbit. 
    Denote by $D$ be the value of the derivative of $\ell^k_\mu$ at any points of $\orbit(\cT)$; the value does not depend on the point chosen. 

    If $D >0$, we call $\orbit(\cT)$ a {\bf regular orbit} and $\cT$ is a {\bf regular (cyclic) trapping region}. 
    In this case $|\cT|=k$.

    If $D < 0$, we call $\orbit(\cT)$ a {\bf flip orbit} and $\cT$ is a {\bf flip (cyclic) trapping region}. In this case $|\cT|=2k$.
    
    The case $D=0$ is degenerate.
\end{definition}

Notice that, since periodic orbits are stable with respect to small perturbations, a trapping region $\cT=\cT_\mu$ depends continuously on $\mu$.
\begin{definition}
    \label{def:range}
    For any trapping region $\cT$, we denote by {\BF$\text{\bf Range}(\cT)=[\mu_0,\mu_1]$} the closed maximal $\mu$ interval on which $\cT=\cT_\mu$ can be defined continuously and write $p_1=p_1(\mu)$ for the periodic point on the boundary of $\J_1=\J_1(\cT_\mu)$.
    We say that $\cT_\mu$ is a 1-parametric family of regular trapping regions when it is a regular cyclic trapping region for every $\mu\in\text{Range}(\cT)$. 
     We say that $\cT_\mu$ is a 1-parametric family of flip trapping regions when it is a regular cyclic trapping region for at least a value $\mu\in(\mu_0,\mu_1)$. 
%     In this case, $\orbit(\cT_{\mu})$
% %    Its range would start at $\mu_0$, where the orbit 
%     has a period-doubling bifurcation point with $D =-1$ at $\mu_0$. At the right endpoint of the window, $c$ falls eventually on $\orbit(\cT_{\mu_1})$.
\end{definition}
{\bf Beginning and end of (a family of) trapping regions.} 

%Notice, first of all, that if $\orbit(\cT_\mu)$ is attracting/repelling/regular/flip for a single value of $\mu$ in $\text{Range}(\cT_\mu)$, then it is so for every value of $\mu$ within $\text{Range}(\cT_\mu)$.
{\bf Let $\cT$ be a regular trapping region.}
When $\orbit(\cT)$ is repelling, the family begins at $\mu=\mu_0$ with an 
%$\orbit(\cT_{\mu_0})$
%    Its range would start at $\mu_0$, where the orbit 
attractor-repellor bifurcation point with $D =+1$. 
An example is shown in Fig.~\ref{fig:p3ic} in case of the period-3 window, where $\mu_0$ and $\mu_1$ are labeled $\pthree{b}$ and $\pthree{e}$, where $b$ and $e$ are for beginning and end. 
Each node has its own $p_1$, $q_1$ and $J_1$ and to distinguish between them we use primes and double primes.
A pair of period-3 orbits $\cO'$, given by $p'_1\mapsto p'_2\mapsto p'_3\mapsto p'_1$, and $\cO$, given by $p_1\mapsto p_2\mapsto p_3\mapsto p_1$, arises at $\pthree{b}=1+2\sqrt{2}\simeq3.828$. 
$\cO'$ is attracting until period-doubling, $\cO$ is always repelling. 
The two orbits coincide at $\pthree{b}$.
Notice that there is no other such pair of period-3 orbits for the logistic map.

There is a period-3 regular trapping region $\cT$ for which $\cO$ lies on its boundary.
We show in red, at $\mu=\pthree1$, the interval $\J_1$ of $\cT$.
%belonging to the period-3 regular trapping region $\cT$.  at whose boundary lie $\cO$. % the repelling period-3 orbit $p_1,p_2,p_3$. 
At $\pthree3\simeq3.846$, we show the interval $\J_1$ above and, written over it in blue, the intervals $\J'_4$ (top), $\J'_1$ (center) and $\J'_7$ (bottom) of a period-9 regular trapping region $\cT'$ that is nested in $\cT$.
The ends of these blue intervals, not labeled in the picture, have black dots.
Notice that the $\J_1$ intervals of all cyclic trapping regions are all symmetric with respect to $c$ and so are all one inside the other; correspondingly, given any two cyclic trapping regions $\cT$ and $\cT'$, the intervals of one are all strictly contained in the intervals of the other.
%, given any two cyclic trapping regions $\cT$ and $\cT'$, either $\J_1(\cT)$ 
%whose intervals are  contained in the intervals of $\cT$. 

The family ends at $\mu=\mu_1$, when $c$ falls eventually on $\orbit(\cT_{\mu})$, namely on $\cO$ (but $c$ does not belong to $\cO$). An example is the case $\mu=\pthree{e}\simeq3.8568$ in Fig.~\ref{fig:p3ic}.

When $\orbit(\cT)$ is attracting, at $\mu=\mu_0$, an orbit, $\orbit(\cT_{\mu_0})$, of period $k$ is created, for some $k$, with $D=+1$. 
Such case is shown in Fig.~\ref{fig:ic2} at $\mu=1$, where the fixed point $p_1$ of $\ell_\mu$ arises.
Caution: if the orbit is created at a period-doubling bifurcation, $k$ means the period of the newly created, period-doubled, orbit. 
The trapping region $\cT$ ends, at $\mu=\mu_1$, at the super-stable point, namely $c$ belongs to $\orbit(\cT_{\mu_1})$. 
Examples are shown in Fig.~\ref{fig:ic2}, at $\mu=2$, and in Fig.~\ref{fig:p3ic}, at $\mu=\pthree{ss}\simeq3.832$.
Strictly speaking, at $\mu=\mu_1$ the trapping region is degenerate since each interval collapses into a point. 

{\bf Let $\cT$ be a flip trapping region.} Such a family always starts at a super-stable point, where the trapping region is degenerate. 
Examples are shown in Fig.~\ref{fig:ic2}, where the attracting fixed point $p_1$ is superstable at $\mu=2$, and in Fig.~\ref{fig:p3ic}, where $p'_1$ belongs to $\cO'$, whose super-stable point is $\pthree{ss}\simeq3.832$.
Close enough to the super-stable point, $\orbit(\cT_\mu)$ is attracting (see the period-2 flip trapping region in Fig.~\ref{fig:ic2}).

At some $\bar\mu\in(\mu_0,\mu_1)$, $\orbit(\cT_\mu)$ has a period-doubling bifurcation point and becomes repelling.
In Fig.~\ref{fig:ic} we show several examples of flip trapping regions with a repelling periodic orbit at their boundary. At $\mu=\pone0$, there is only one such trapping region $\cT=\{\J_1,\J_2\}$. 
Here $\J_1=[q_1,p_1]$ and $\J_2=[p_1,q_2]$, where $p_1$ is the flip fixed point of the logistic map, $q_1=1-p_1$ and $\ell_\mu(q_2)=q_1$. Both $\J_1$ and $\J_2$ are painted in red. 
At $\mu=\pone1$, a second flip trapping region $\cT'=\{\J'_1,\dots,\J'_4\}$ arises, where $\J'_1=[p_1',q'_1]$, $\J'_2=[p_2',q'_2]$, $\J_3=[q'_3,p_1]$ and $\J'_4=[q'_4,p'_2]$. 
In this case, $p'_1,p'_2$ is a period-2 orbit, $\ell_\mu^2(q'_3)=q'_1$ and $\ell_\mu^2(q'_4)=q'_2$. %At the boundaries of the $\J'_j$ lie the period-2 orbit $p'_1,p'_2$. 
This trapping region is painted in blue and its intervals $\J'_i$ are proper subsets of the $\J_i$.

At $\mu=\pone2$, a third flip trapping region $\cT'=\{\J'_1,\dots,\J'_8\}$ arises, where $\J''_1=[q''_1,p''_1]$ and so on. The $\J''_i$ are painted in dark green. At their boundaries lie the period-4 orbit $p''_1,\dots,p''_4$. In this case, $\ell_\mu^4(q'_5)=q'_1$ and so on.
Another example is provided in Fig.~\ref{fig:p3ic} at $\mu=\pthree2$. The intervals $\J_1',\J'_4$ shown, which are part of a period-6 flip trapping region of $\ell_\mu$, are a period-2 flip trapping region for $\ell^3_\mu$. 

The family ends, at $\mu=\mu\INT$, when $c$ falls eventually on $\orbit(\cT_{\mu})$ (but it is not super-stable). In this special case, all endpoints of the $\J_i$ are on the orbit of $c$.
In Fig.~\ref{fig:ic2} we mark with a cyan vertical line the point 
\beq
\label{Merge}
\mu_M=3[ 1 + (19-3\sqrt{33})^{1/3} + (19+3\sqrt{33})^{1/3} ]/2\simeq3.67857.
\eeq
This value marks the end of the flip trapping region $\cT$ with $\orbit(\cT)=\cO$. 
At $\mu_M$, we have that $\ell_\mu(c)=q_2$, $\ell^2_\mu(c)=q_1$, and $\ell_\mu^3(c)=p_1$. 
In Fig.~\ref{fig:p3ic} we do the same with the point $\pthree{M}\simeq3.851$. This value marks the end of the flip trapping region $\cT'$ such that $\orbit(\cT')=\cO'$.
%at whose boundary lie the period-3 orbit $p'_1,p'_2,p'_3$.

\begin{figure}
 \centering
 \includegraphics[width=11cm]{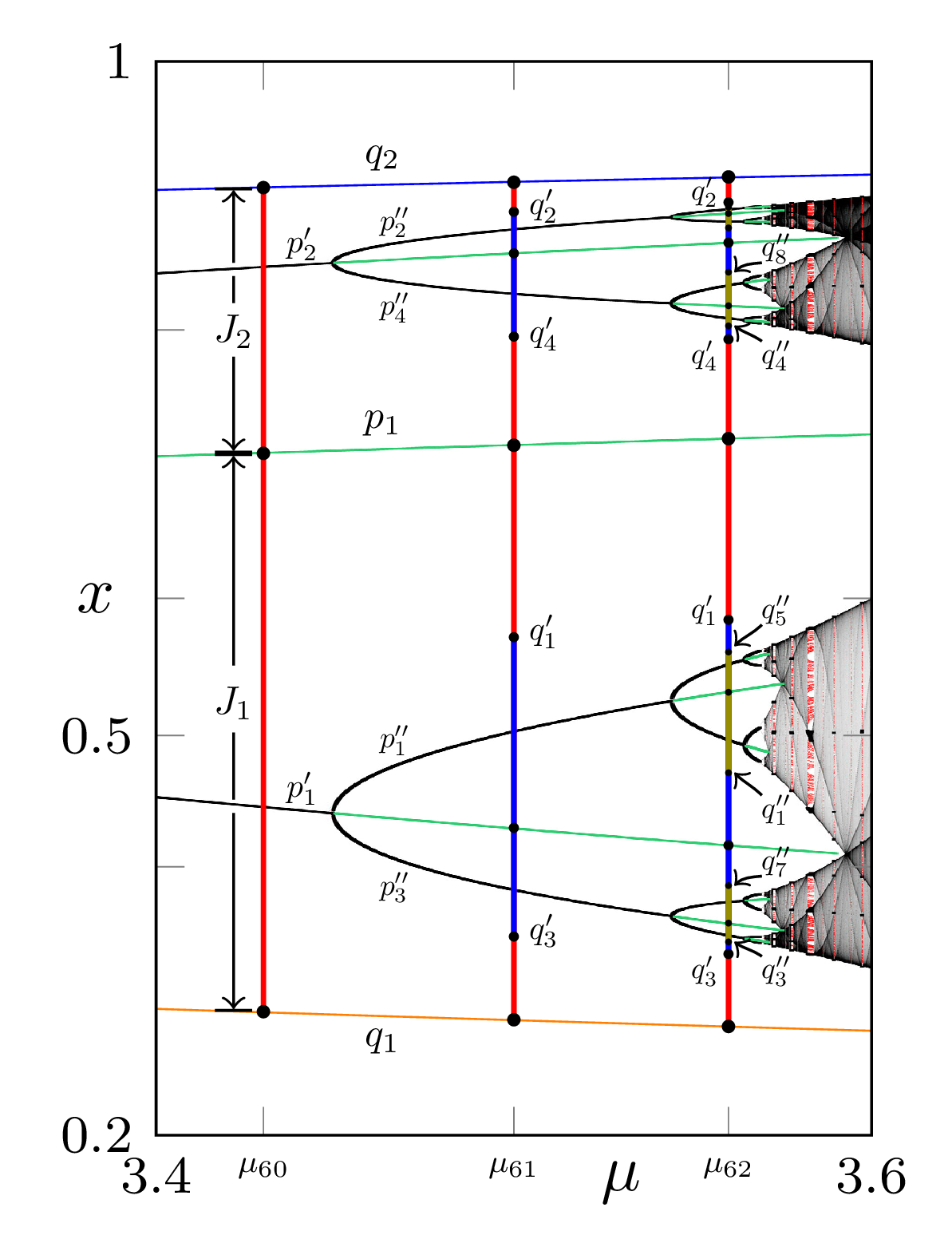}
 \caption{{\bf Flip trapping regions of the logistic map.} 
 This picture is a blowup of Fig.~\ref{fig:ic2}. It shows at the three parameter values $\pone0=3.43,\pone1=3.5,\pone2=3.56$, all flip cyclic trapping regions of the logistic map with a repelling orbit at their boundary. 
 At $\mu=\pone0$, there is a single flip trapping region $\cT=\{\J_1,\J_2\}$, where $\J_1=[q_1,p_1]$ and $\J_2=[p_1,q_2]$. The point $p_1$ is the flip fixed point, $q_1=1-p_1$ and $\ell_\mu(q_2)=q_1$. 
 A description for the other two values of $\mu$ is given in the text below Def.~\ref{def:range}.
 }
 \label{fig:ic}
\end{figure} 

\begin{figure}
 \centering
 \includegraphics[width=14.4cm]{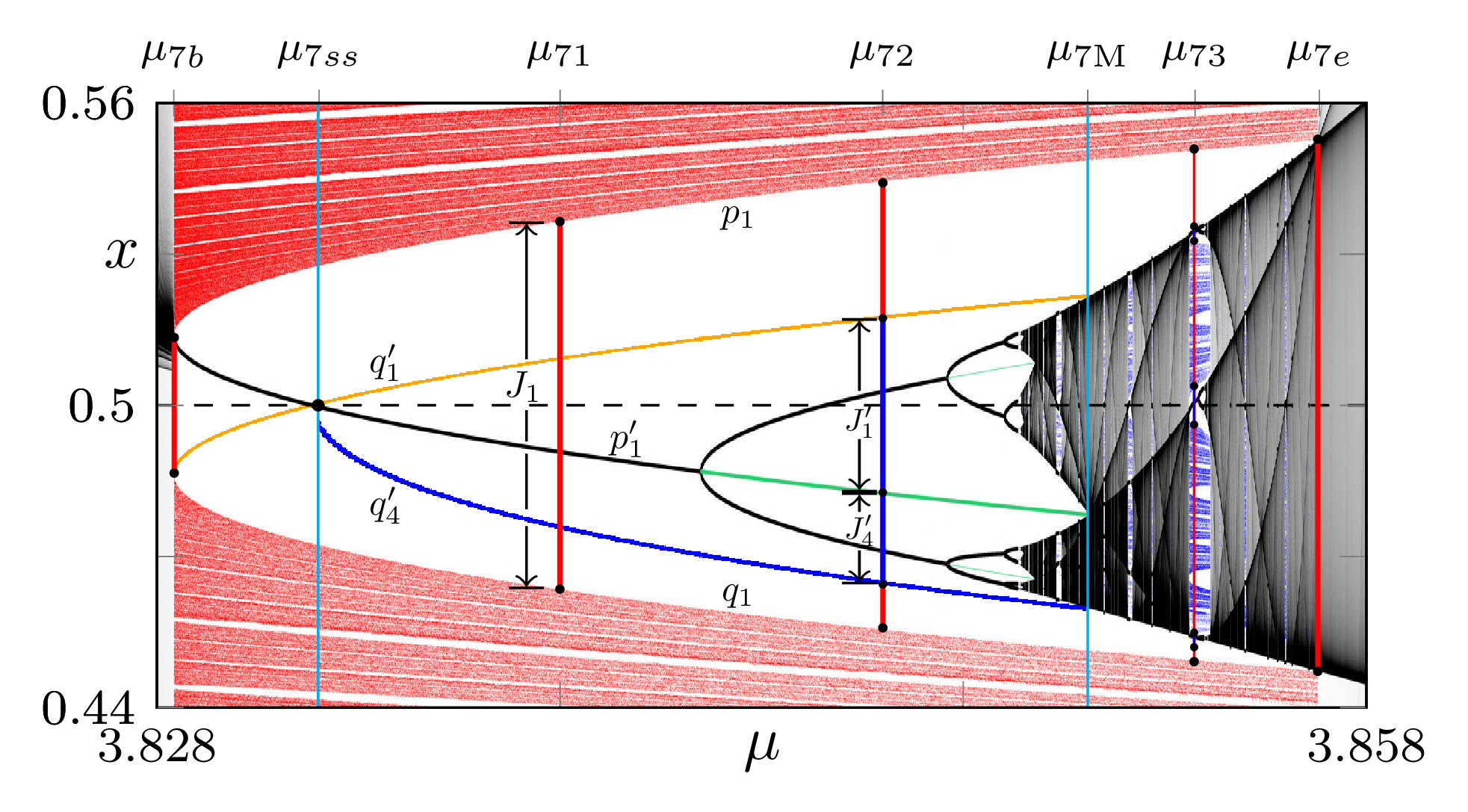}
 \caption{{\bf Examples of regular and flip cyclic trapping regions.} 
    In this detail of the period-3 window, the curve of points denoted by $p_1$ is a period-3 repellor. It separates the red Cantor set from the white basin of the attractor. The curve of points $q_1=1-p_1$ separates the Cantor set from the basin. 
    A description for the marked values of $\mu$ is given in the text below Def.~\ref{def:range}.
}
 \label{fig:p3ic}
\end{figure} 
\bigskip

We now define the well-known concept of a window in a bifurcation diagram.
We define it from an unusual point of view which emphasizes the importance of trapping regions. 
A window is where a particular kind of trapping region exists. 
\begin{definition}
    \label{def:win}
    Let $\cT$ be a period-$k$ regular trapping region such that $\orbit(\cT)$ is repelling.
    Let the $\mu$-interval $W$ be the (maximal) range of $\cT$. 
    We say that $W$ is a  {\bf period-\BF$k$ window}.
    When a period-$k_1$ window $W_1$ contains a period-$k_2$ window $W_2$, with $k_2\geq k_1$, we say that $W_2$ is a {\bf subwindow} of $W_1$. Notice that $k_2$ is a multiple of $k_1$.
\end{definition}
We will show in Prop.~\ref{prop:k'=Ck} that, if $W_1=\text{Range}(\cT_1)$ and $W_2=\text{Range}(\cT_2)$ and both $\orbit(\cT_1)$ and $\orbit(\cT_2)$ are repelling, then $k_2>k_1$.

\bigskip\noindent{\bf Nodes.} Here we will study in detail nodes and their relation with cyclic trapping regions. 
We will show that each node is paired with a cyclic trapping region and this cyclic trapping region has always a periodic orbit of the node on its boundary. 
This will enable us to classify all repelling and attracting nodes of the logistic map.

\begin{figure}
 \centering
 \includegraphics[width=13.5cm]{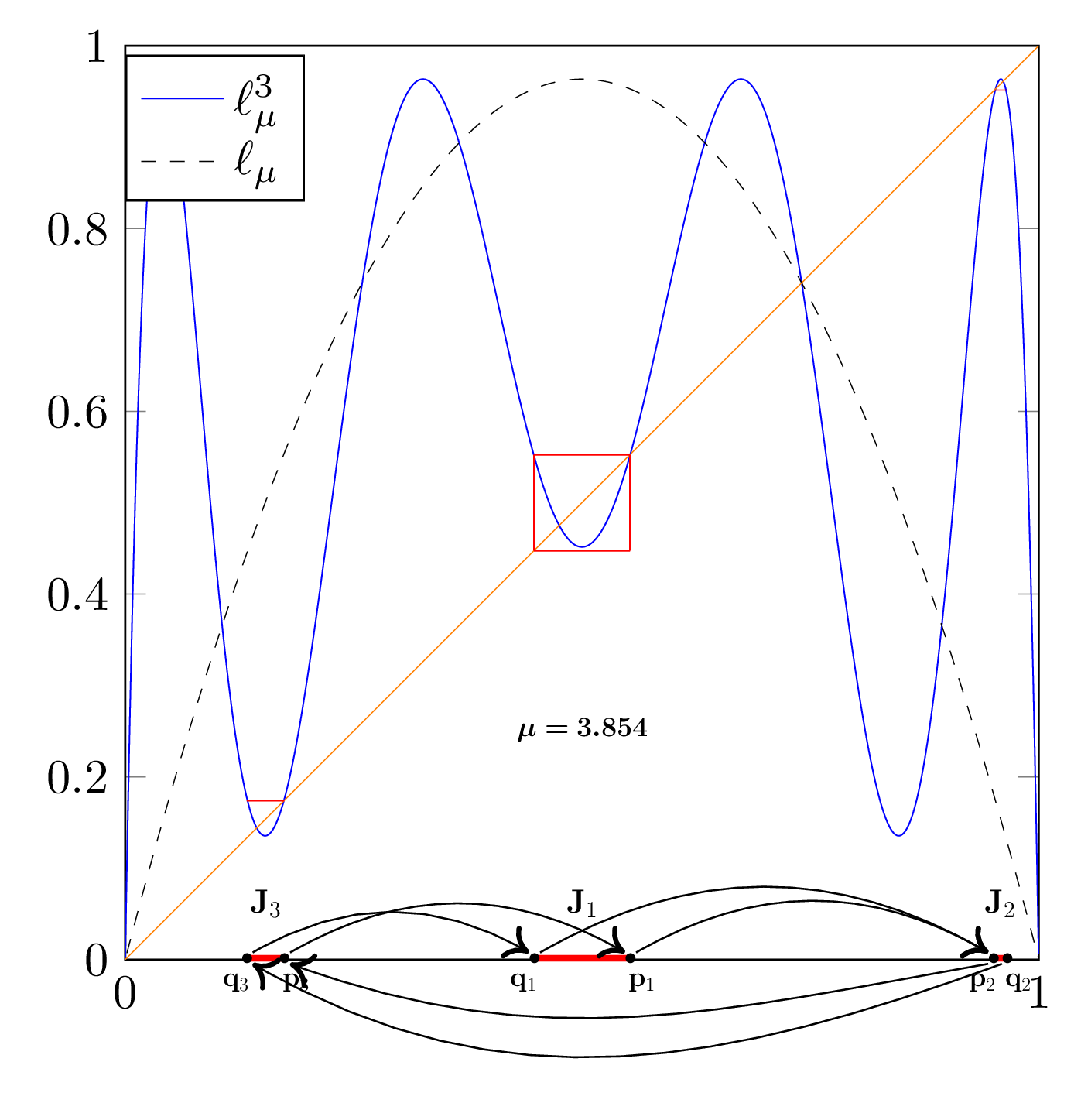}
  \caption{
  {\bf A regular cyclic trapping region} for the logistic map. The bifurcation diagram for the logistic map has a period-3 window in parameter space starting at $\mu_0=1+\sqrt{8}\simeq3.828$, where a pair of an attracting and a repelling period-3 orbits arise, and ending at $\mu_1\simeq3.857$, where the unstable periodic orbit collides with the attractor (i.e. there is a {\em crises}~\cite{GOY82}).
  Here we show what is happening at one of the intermediate parameter values, { $\mu =3.854$.}
  %$\mu =3.8584$
  There are intervals $\J_i$ with endpoints $q_i,p_i$, $i=1,2,3$, shown in red, which, together, form a period-3 regular cyclic trapping region for $\lmu$ (Def.~\ref{def:ic}). The arrows show how the endpoints map under $\lmu.$
  The intervals are chosen so that 
  $p_1\mapsto p_2\mapsto p_3\mapsto p_1$ is an unstable period-3 orbit
  and $\lmu^3(q_i)=p_i$ for $i=1,2,3$. This construction gives rise to a cyclic trapping region because, for this value of $\mu$ and this choice of the endpoints, $\lmu^3(\J_i)\subset \J_i$. 
  The picture also includes the graph of $\lmu^3(x)$ and $\lmu(x)$.
%  The solid graph however is $\lmu^3(x)$ while the dashed graph is $\lmu(x)$.
  }
  \label{fig:pric}
\end{figure}

\begin{figure}
 \centering
 \includegraphics[width=7cm]{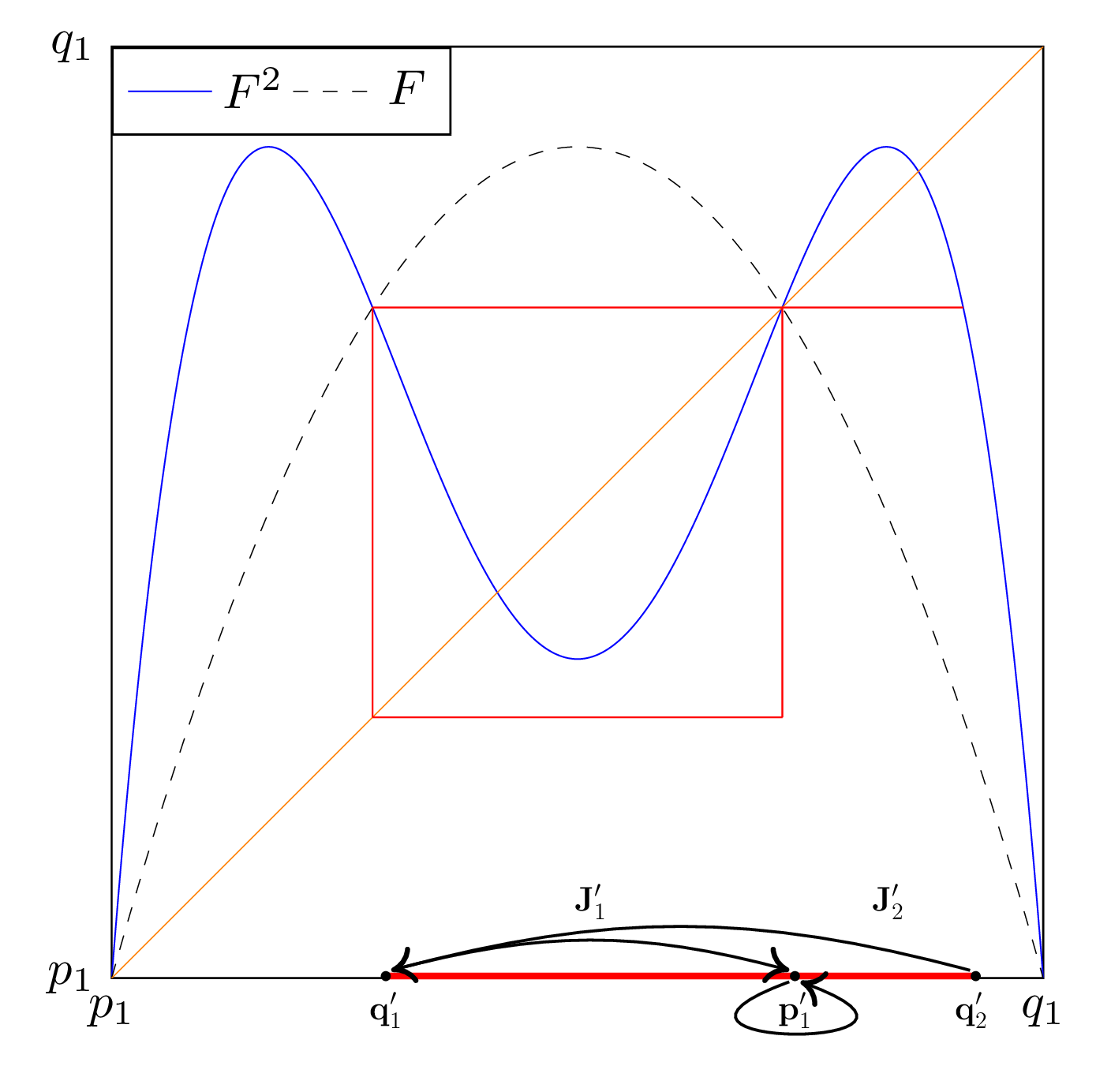}
 \includegraphics[width=7cm]{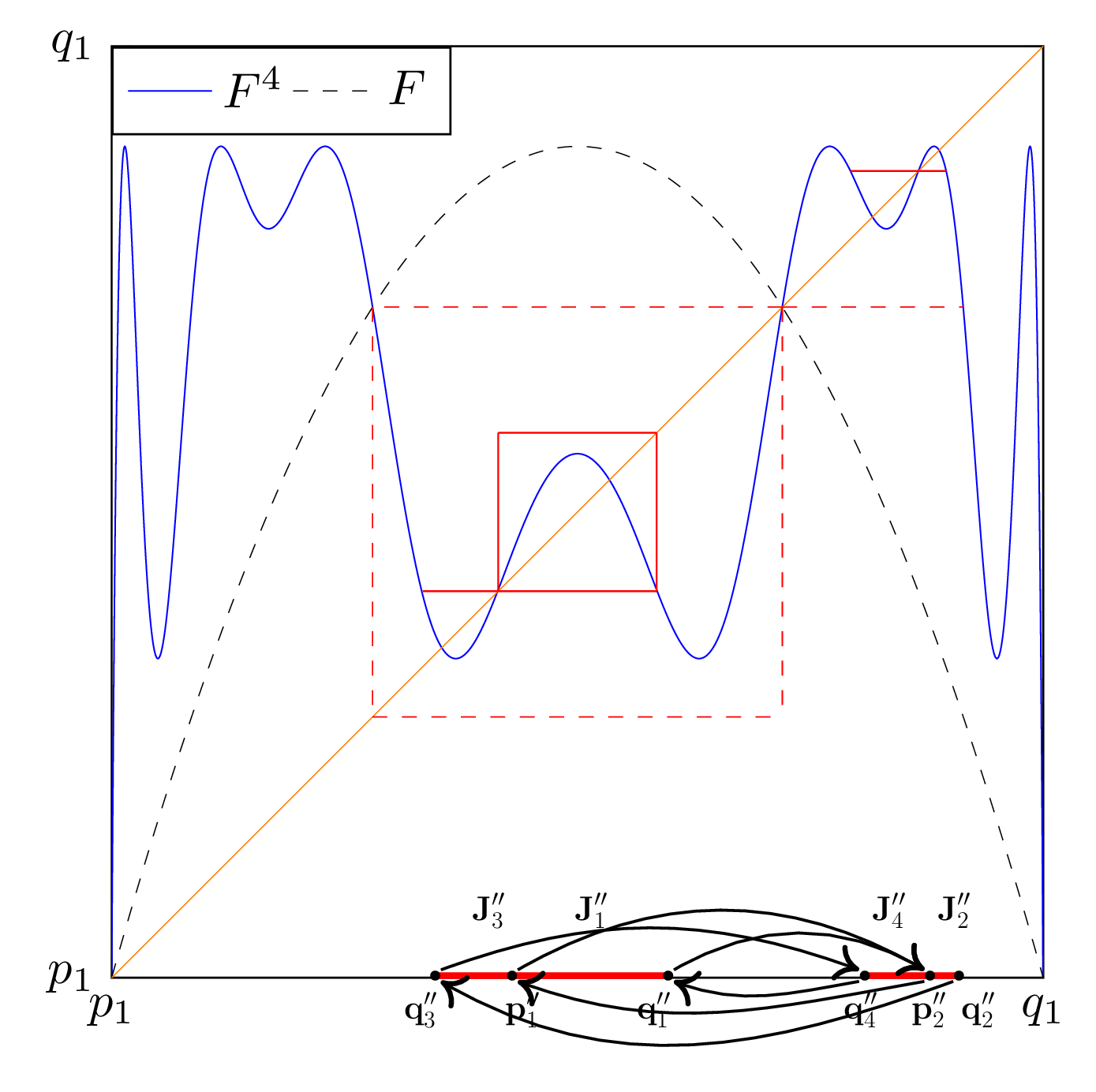}
  \caption{
  {\bf Nested flip cyclic trapping regions} for the logistic map.
  On the left panel, we show an interval $[p_1,q_1]$ that is the $\J_1$ interval of a period-$k$ cyclic trapping region $\cT$ for $F=\ell_\mu^k$. Inside $\J_1$, we show a flip cyclic trapping region (Def.~\ref{def:ic}) $\cT'=\{\J'_1,\J'_2\}$, where $\J_1'=[q_1',p_1']$ and $\J'_2=[p'_1,q'_2]$.
%  that belongs to a trapping region $\cT'$ nested in $\cT$.
%  for $F^2$.
  On the right panel, again inside $\J_1$, we show a flip cyclic trapping region $\cT''=\{\J_1'',\J_2'',\J_3'',\J''_4\}$ nested in $\cT'$, where 
  $\J''_1=[p_1'',q_1'']$, $\J''_2=[p''_2,q''_2]$, $\J''_3=[q''_3,p''_1]$ and $\J''_4=[q''_4,p''_2]$.
  The arrows in both panels show how the endpoints map under $F$. On the left, the periodic orbit at the boundary of the cyclic trapping region is the fixed point $p'_1$, on the right is the period-2 orbit $p''_1\mapsto p''_2\mapsto p''_1$.
  Notice that $F^2(\J'_1)\subset \J'_1$ and $F^4(\J''_1)\subset \J_1''$.
  The picture also includes the graphs of  $F(x)$,  $F^2(x)$ and $F^4(x)$.
  The actual value used in these pictures is $\mu=\mu_{FM}$; at this value, there is an infinite sequence of flip cyclic trapping regions nested one into the other.
  }
  \label{fig:p24ic}
\end{figure}

\bigskip

\begin{definition}
    Given a node $N$ of $\ell_\mu$, we denote by $\rho(N)$ its minimum distance from the critical point $c$. We write $N_1>N_2$ if $\rho(N_1)>\rho(N_2)$. 
\end{definition}
Notice that $\rho(N)>0$ for every repelling node. In contrast, $\rho(N)=0$ only if $N$ contains a chaotic attractor or an almost periodic attractor or a superstable periodic attractor.

We will show in the remainder of the section that $N_1>N_2$ implies that $N_1$ is upstream from $N_2$. 
%In Prop.~\ref{prop:upstream} we show that $N_1\geq N_2$ iff $N_1$ is upstream from $N_2$. 
For the logistic map, this is ultimately equivalent to the fact that the graph is a tower.
We start by showing that $\rho$ is injective.
\begin{proposition}
    \label{prop:rho}
    Let $N_1$ and $N_2$ be two distinct nodes. 
    Then either $N_1>N_2$ or $N_2>N_1$. 
%    If     $\rho(N_1)>\rho(N_2)$, then $N_1$ is upstream from $N_2$.
%    The relation $N_1\geq N_2$ iff $\rho(N_1)\geq\rho(N_2)$ defines
%    a linear order on the set of nodes.
\end{proposition}
\begin{proof}
    Let $p_0$ be a point of $N$ such that $|p_0-c|=\rho(N)$. The other point having distance $\rho(N)$ from $c$ is $1-p_0$. 
    Since $\ell_\mu(1-p_0)=\ell_\mu(p_0)\in N$, then
    $1-p_0$ is either in the node or in its preimage and therefore it cannot belong to any other node. 
    Hence, no two nodes are equidistant from $c$.
%    Clearly $N_1\geq N_1$ and, from $N_1\geq N_2$ and
%    $N_2\geq N_3$, follows that $N_1\geq N_3$.
    
%    Now, let $p_0$ be a point of $N$ such that $|p_0-c|=\rho(N)$.
%    Then, either $1-p_0$ or $\ell_\mu(1-p_0)$ belong     to $N$ too, since $\ell_\mu(p_0)=\ell_\mu(1-p_0)$, and so $1-p_0$ cannot belong to any other node. 
%Hence every node has a different distance from $c$
%and so $N_1\geq N_2$ and $N_2\geq N_1$ implies that $N_1=N_2$. 
%    Moreover, for any two nodes $N_1,N_2$, clearly either $N_1\geq N_2$ or $N_2\geq N_1$.
\end{proof}
%
%Since no node has the same distance from $c$, this distance determines a linear order relation between nodes, namely we say that $N_1\geq N_2$ for $\rho(N_1)\geq\rho(N_2)$. 

%In turn, this immediately implies that the graph of the logistic map is a tower.
%
\begin{definition}
    \label{def:p0}
    Let $N$ be a node of $\ell_\mu$ with $\rho(N)>0$. From now on, {\BF$p_0=p_0(N)$} will refer to a point in $N$ with minimal distance from $c$.
%   a point such that $|p_0-c|=\rho(N)$. 
    We denote by {\BF $\J_1(N)$} the closed interval with endpoints $p_0$ and $1-p_0$. 
\end{definition}
We will later show that one and only one of the two is periodic and we will later refer to the periodic one as $p_1$. 
In Figs.~\ref{fig:ic} and~\ref{fig:p3ic} we show sets $\J_1$, together with the corresponding endpoint $p_1$, for several nodes.
\begin{proposition}[Downstream Proposition]
  \label{prop:downstream}
  Let $N$ be a node of $\ell_\mu$ with $\rho(N)>0$. Then each
%  node in $\Jall(N)$ 
  chain-recurrent point in $J_1(N)$
  is downstream from $N$.
%    contains no point that eventually maps into $N$ under $\ell_\mu$. 
%  When the attractor is a periodic orbit, then a single point $q_0$ of the attractor is contained in $\J_1(N)$ and $\J_1(N)$ is the component of the immediate basin of the attractor containing $c$.
\end{proposition}
\begin{proof}

Assume, for discussion sake, that $p_0<c$; the argument for $p_0>c$ is virtually the same. We set $K=K_\varepsilon=[p_0,p_0+\varepsilon]$. 

There are two cases. 

CASE 1: Assume that for every $\varepsilon>0$, 
for some $r>1$, $\ell^r_\mu(K_\varepsilon)$ contains $c$ in its interior.
%, while $\ell^r_\mu(p_0)$ is not in that interior.

Since $p_0$ is the closest point of $N$ to $c$, $\ell^r_\mu(p_0)$ cannot be in the interior of $\J_1(N)$. Then $\ell^r_\mu(K_\varepsilon)$ must contain at least either $[p_0,c]$ or $[c,1-p_0]$. 
Denote by $\J_0$ the half of $\J_1$ that $K_\varepsilon$ eventually maps onto for arbitrarily small $\varepsilon$. 
Hence, for arbitrarily small $\varepsilon$, there are $\varepsilon$-chains from $p_0$ to each point within $\J_0$. Notice that chain-recurrent points $x$ and $1-x$ belong to the same node and that, if $p_0$ is upstream from any point of a node, is upstream from each point of that node.

%%If a point $x$ eventually maps into $N$, then so does $1-x$, hence it is sufficient to show that either the left half or the right half contain no preiterates of $N$ since $\ell_\mu$ is symmetric with respect to $c$.
%Hence either $x$ or $1-x$ is both upstream and downstream from $N$, so it belongs to $N$. This is a contradiction since $p_0$ is the closest point of $N$ to $c$.
%As stated above, for some $r>1$, $\ell^r_\mu(K_\varepsilon)$ contains $c$ in its interior, while $\ell^r_\mu(p_0)$ is not in that interior. 
%It follows that no point in the interior of $K_\varepsilon$ can map onto $p_0$ and $1-p_0$. 
%So $\ell^r_\mu(p_0)=p_0$ or $\ell^r_\mu(p_0)=1-p_0$. It follows that $p_0$ or $1-p_0$ is periodic. 

%{\allred
%Remark 1: It follows that when the attractor is not periodic, $\J_1$ is invariant.
%}

CASE 2: Assume that, for some $\varepsilon>0$ and
%(in fact, for all $\varepsilon$ sufficiently small), 
all $r>1$, $\ell^r_\mu(K_\varepsilon)$ does not contain $c$ in its interior. In particular, the attractor must be a periodic orbit by the Homterval Theorem.

If $N$ contains non-periodic points, let $p$ be one of them and assume, by discussion sake, that $p<c$.
%non-periodic point of $N$ to the left of $c$. 
Then $[p,p+\varepsilon]$ cannot be a homterval and so, for the same argument above, every point of $\J_1(N)$ is downstream from $N$. 
%Hence there cannot be points that are upstream from $N$ in the interior of $\J_1$.

If all points of $N$ are periodic, then $p_0$ belongs to a period-$k$ orbit and $K_\varepsilon$ lies in a maximal homterval we denote by $H$. 
So $H$ is invariant under $\ell_\mu^{2k}$, which is orientation preserving. 
If $h\in H$ and it is not a fixed point for $\ell_\mu^{2k}$, then $h$ converges to an attractor, and there is only one attractor for $\ell_\mu$. 
So the attractor is a periodic-$2k$ orbit. 
By Singer theorem, there is an interval $H_c$ in the basin of this orbit that contains both $c$ and a point $q_0$ of the period-$2k$ orbit. 
We can assume $H_c$ is the largest open interval that is in the basin and contains $c$. 
From the construction, the boundary of $H_c$ contains a period-$k$ point $p^*$. 
The other boundary point is $1-p^*$. So $p^*$ is the closest point of $N$ to $c$. Hence, either $p^*=p_0$ or $p^*=1-p_0$ and so $\J_1(N)=H_c$.

In particular, the open interval with endpoints $q_0$ and $1-q_0$ is a subset of the basin of attraction of the attractor and therefore does not contain any chain-recurrent point. Hence, by the same argument above, all chain-recurrent points in $\J_1(N)$ are downstream from $N$.
%is a homterval. $K$ lies in a maximal interval which is a homterval. We will assume that $K$ is that interval. Since $p_0$ is recurrent, the image of $K$ intersects itself under $\ell_\mu$. 
\end{proof}
{
The proof above implies, in particular, the following:
%{\allred add Theorem Corollary}
%
% \allred
% \begin{corollary}
%   Suppose that the attractor of $\ell_\mu$ is a period-$2k$ orbit, that the repelling node $N$ is also a periodic orbit and that the interval $[p_0,p_0+\varepsilon]$ is a homterval. Then in the interior of $\J_1(N)$ there is a single point $q_0$ of the attractor and the interior of $\J_1(N)$ is the component of the immediate basin containing $q_0$.
% %  {\allred actually, most likely there is a cyclic trapping region, with $\J_1=(q_0,1-q_0)$, associated to the attractor! periodic attractors are always the periodic orbits at the boundary of a cyclic trapping region}
% \end{corollary}
% }
%
\begin{corollary}
    Let $N$ be a node of $\ell_\mu$ with $\rho(N)>0$. Then   there is no point in the interior of $\J_1(N)$ that falls eventually into $N$.
\end{corollary}
\begin{proof}
    A point falling eventually on $N$ is upstream from $N$. By Prop.~\ref{prop:downstream}, all chain-recurrent points of $\J_1(N)$ are downstream from $N$, so a point in the interior of $\J_1(N)$ falling eventually into $N$ would be both upstream and downstream from $N$ and so it would belong to $N$. This is not possible because, by hypothesis, $p_0$ is the closest point of $N$ to $c$.
\end{proof}
%Next theorem shows that to every node $N$ is naturally associated a cyclic trapping region such that %
%Recall that, in Prop.~\ref{thm:ic2node}, we showed that every cyclic trapping region $\cT$ has a periodic orbit, denoted by $\orbit(\cT)$, at its boundary. 
%
\begin{definition}
    \label{def:N(T)}
    We denote by $\text{Node}(\cT)$ the node that contains $\orbit(\cT)$.
\end{definition}
The node 
$\text{Node}(\cT)$ contains no points of $\Jall(\cT)$, in fact it is the node closest to $c$ with this property.  Notice that $\rho(\text{Node}(\cT))>0$ unless $\orbit(\cT)$ is attracting and $c$ lies on it.
%{\allred The periodic unstable orbit seems to be the companion of the stable periodic orbit the cyclic trapping region arose from.}
%In Prop.~\ref{thm:node2ic} we will show that this map from cyclic trapping regions to nodes 
%determined in Thm.~\ref{thm:ic2node} 
%is invertible. 
%The unique periodic orbit on the boundary of the $\J_i$  is a subset of the node; in some cases the orbit is the entire node. 
We are now in position to show that this map from cyclic trapping regions to nodes with strictly positive distance from $c$
%defined in Thm.~\ref{thm:ic2node} 
can be inverted.

%In the following theorem, $p_0=p_0(N)$ as defined in Def.~\ref{def:p0}. Of course $\ell_\mu^k(p_0)=\ell_\mu^k(1-p_0)$ for all $k\geq1$.
%
\begin{proposition}[The cyclic trapping region of a node]
\label{thm:node2ic}
  Let $N$ be a node of $\ell_\mu$ with $\rho(N)>0$. 
  Let $p_0=p_0(N)$. %(Def.~\ref{def:p0}).
  Then:
  %there exist an integer $r\geq1$ such that:
%  the two following equivalent properties hold:
 \begin{enumerate}
 
 \item Either $p_0$ or $1-p_0$ is a periodic point. Let $k$ denote its period.
 
 \item There is a cyclic trapping region $\{\J_1,\dots,\J_r\}$, where $\J_1=\J_1(N)$, whose period $r$ is either $k$ or $2k$.

% 1. $\ell_\mu^r(\J_1(N))\subset \J_1(N)$.
    
%\noindent   Let $k$ be the smallest such value.
%    There is a period-$k$ cyclic trapping region $\{\J_1,\dots,\J_k\}$ with $\J_1=\J_1(N)$.

%    2. Either $p_0$ or $1-p_0$ is a periodic point with period $k$ or $k/2$.
%  and $\orbit(\cT(N))\subset N$. 
  \end{enumerate}
  
%  For the smallest such $k$, 
%  \noindent There is a period-$k$ cyclic trapping region $\{\J_1,\dots,\J_k\}$ with $\J_1=\J_1(N)$.
\end{proposition}
Notice that $\ell_\mu^r(p_0)=\ell_\mu^r(1-p_0)$ for all $r\geq1$.
%  Remark: for the smallest such $k$, there is a period-$k$ cyclic trapping region $\{\J_1,\dots,\J_k\}$ with $\J_1=\J_1(N)$.
%Note that, in particular, $o(\cT_N)$ belongs to $N$ and the minimum distance is achieved on it.
%
\begin{proof}
Let $\J_0=\J_1(N)\cap[0,c]$ and denote by $\inter{J}_1$ the interior of $\J_1(N)$.

Assume first that the attractor is chaotic or almost periodic. Then there cannot be homtervals and, by the Homterval Theorem,  there is some $k\geq1$ such that $\ell_\mu^k(\J_0)=\ell_\mu^k(\J_1)$ contains $c$ in its interior. 
In other words, there is a point $x$ in $\inter{J}_1$ such that $\ell^k_\mu(x)=c\in\inter{J}_1$.

We claim that this is enough to grant that $\inter{J}_1$ is invariant under $\ell^k_\mu$. Indeed, by continuity, if there were a $y\in\inter{J}_1$ with $\ell^k_\mu(y)\not\in\inter{J}_1$, there would be some $\xi\in\inter{J}_1$ between $x$ and $y$ such that either $\ell^k_\mu(\xi)=p_0$ or $\ell^k_\mu(\xi)=1-p_0$. This, though, is impossible because, by our Downstream Proposition, there are no preiterates of $N$ in $\inter{J}_1$.
Moreover, by continuity, the common value of $p_0$ and $1-p_0$ under $\ell^k_\mu$ must belong to $\J_1(N)$ and, since $N$ is forward invariant and there are no points of $N$ in $\inter{J}_1$, it can only be either $p_0$ or $1-p_0$, so the theorem is proved in this case. 
%the graph of $\ell_\mu^k$ intersects the interior of the square $Q_1=\J_1(N)\times \J_1(N)$. 
%We claim that the graph intersects the boundary of $Q_1$ at its corners. 
%Indeed, by the lemma above, in the interior of $\J_1(N)$ there are no preiterates of $N$, so the graph cannot intersect the interior of the two horizontal boundary components of $Q_1$. Moreover, $p_0\in N$ and $\ell_\mu(1-p_0)=\ell_\mu(p_0)$, so all iterates of both $p_0$ and $1-p_0$ belong to $N$. %Since  there is no point of $N$ in the interior of $\J_1(N)$, then the graph of $\ell^k_\mu$ cannot intersect the interior of the vertical boundary components of $Q_1$ either. 
%Hence it must intersect $Q_1$ at two of its four corners and, because of the symmetry of $\ell_\mu$, it will intersect either at its lower or upper corners. 

Assume now that the attractor is a period-$r$ orbit. Then (see the corollary above) there is a point $q_0$ of the attractor in $\J_1$. Hence $\ell_\mu^{r}(q_0)=q_0$ and so, as above, there is a point in $\inter{J}_1$ whose image under $\ell_\mu^{r}$ belongs to $\inter{J}_1$.
%{\allred How to justify this when $N$ is a Cantor set?} 
%Hence, the graph of $\ell_\mu^r$ has a point inside $Q_1$. 
%In case $N$ is a period-$s$ orbit, notice that $\ell_\mu^{rs}$ keeps both $p_0$ and $q_0$ fixed. 
%Hence, as above, there is a point of $\inter{J}_1$ whose image under $\ell_\mu^{rs}$ is in $\inter{J}_1$
%Hence the graph of $\ell_\mu^{rs}$ intersects the interior of $Q_1$ 
%and so, b
By the argument above, $\ell_\mu^{r}$ leaves both $\inter{J}_1$ and $J^1(N)$ invariant, proving the theorem in this second case. 

We have left it to the reader to show that the intervals $J_1,\ell_\mu(J_1),\dots,\ell^{r-1}_\mu(J_1)$ have disjoint interiors.

\end{proof}
\begin{definition}
   We denote by {\BF$\cT(N)$} the trapping region with $\J_1=\J_1(N)$. 
\end{definition}
Notice that $\cT(N)$ is the largest cyclic trapping region that does not contain $N$ in its interior.

A byproduct of Thm.~\ref{thm:node2ic} is that either $p_0$ or $1-p_0$ is a periodic point: 
\begin{corollary}
  The minimum distance between a repelling node $N$ and $c$ is achieved at the periodic point $p_1(\cT(N))$.
\end{corollary}
We sometimes denote that periodic point by {\BF $p_1(N)$}.
In Fig.~\ref{fig:ic} we show all cyclic trapping regions corresponding to three parameters values at the left of the Feigenbaum-Myrberg point. 
The nodes corresponding to each trapping region are flip periodic orbits and therefore the trapping regions are flip cyclic trapping regions.

At $\pone0$, there are two repelling nodes $N_0$ and $N_1$, namely the fixed point $0$ and the fixed point $p_1$, and an attracting node $N_2$ consisting of a period-2 orbit $p'_1=p_1(N_2),p'_2$. 
Since $p_1$ is a flip orbit, $\cT(N_1)=\{\J_1,\J_2\}$ has period 2. 
The picture shows, in red, the intervals $\J_1=\J_1(N_1)$, with endpoints $p_1=p_1(N_1)$ and $q_1=1-p_1$, and $\J_2$, with endpoints $p_1$ and $q_2$, which is the larger root of
$\ell_{\pone0}(x)=q_1$.
%such that between it and $p_1$ there is a single critical point of $\ell_{\pone0}^2$.

As $\mu$ increases from $\pone0$ to $\pone1$, the orbit $p_1',p_2'$ bifurcates and becomes repelling, so that at $\pone1$ there is now an attracting node $N_3$ consisting of a period-4 orbit $p''_1=p_1(N_3),\dots,p''_4$. 
The picture shows, again in red, the intervals of $\cT(N_1)$. Moreover it shows, in blue, the intervals $\J'_1=\J_1(N_2),\dots,\J'_4$ of $\cT(N_2)$.
Similarly, as $\mu$ increases from $\pone1$ to $\pone2$, the period-4 orbit bifurcates and becomes repelling, so that at $\pone2$ there is now a period-8 attracting orbit. Here the picture shows, again in red and blue respectively, the intervals of $\cT(N_1)$ and $\cT(N_2)$. Moreover it shows, in olive green, the intervals $\J''_1=\J_1(N_3),\dots,\J''_8$ of $\cT(N_3)$.

In Fig.~\ref{fig:p3ic}, we show some example of regular cyclic trapping regions for some parameter values in the period-3 window. 
Notice that the picture shows only a detail close to the central cascade of the window.

At $\pthree1$, there are two repelling nodes, $N_0=\{0\}$ and $N_1$, the red Cantor set (see Figs.~\ref{fig:full},~\ref{fig:p3}), and an attracting one consisting in a period-3 orbit. 
The picture shows, in red, the interval $\J_1=\J_1(N_1)$ of the regular period-3 cyclic trapping region $\cT(N_1)$, whose endpoints are $p_1=p_1(N_1)$ and $q_1=1-p_1$.
While increasing from $\pthree1$ to $\pthree2$, the period-3 orbit bifurcates so that at $\pthree2$ we have now an attracting node $N_3$ consisting of a period-6 orbit. 
The picture shows the interval $\J_1(N_1)=[q_1,p_1]$ and the intervals $\J_1'=\J_1(N_2)=[p'_1,q'_1]$, where $p'_1=p_1(N_2)$ and $q'_1=1-p_1'$, and $\J'_4=[q'_4,p'_1]$ of the period-6 cyclic trapping region $\cT(N_2)$.
%The endpoints of $\J_1'$ are $p'_1=p_1(N_2)$ and $q'_1=1-p'_1$.

Finally, at $\pthree2$, there are three repelling nodes $N_0,N_1,N_2$, namely the fixed point 0, the red Cantor set and the blue Cantor set, and an attracting node $N_4$ consisting of a period-9 orbit. The picture shows the interval $\J_1(N_1)=[p_1,q_1]$ of the regular period-3 cyclic trapping region $\cT(N_1)$ and, in blue, the intervals $\J'_1=\J_1(N_2)$, not labeled in figure, $\J'_4,\J'_7$ of the regular period-9 cyclic trapping region $\cT(N_2)$.

\allblack

%Recall that the node associated to a cyclic trapping region $\cT=\{\J_1,\dots,\J_k\}$ is the unique node  having a periodic orbit $\orbit(\cT)$ such that every $\J_i$ has an endpoint belonging to it.

Since Prop.~\ref{thm:node2ic} shows that $\orbit(\cT_N)\subset N$, we have the following.
\begin{corollary}
\label{cor:pairing}
    Let $N$ be a node with $\rho(N)>0$. Then $\text{Node}(\cT_N)=N$.
\end{corollary}
%
%\begin{corollary}
%  The minimum distance between a repelling node $N$ and $c$ is achieved at a periodic point and $\text{Node}(\cT_N)=N$.
%\end{corollary}
%
%The existence of a bijection between the cyclic trapping regions and the nodes has two non-trivial consequences:
%
%\begin{corollary}
Notice that, in particular, since every cyclic trapping region $\cT$ is characterized by the periodic orbit at its boundary, the number of cyclic trapping regions of a given $\ell_\mu$ is bounded from above by the number of periodic orbits of the map. While this bound is trivial for $\mu\geq\mu_{FM}$, where there are always infinitely many periodic orbits, it is actually sharp for $\mu<\mu_{FM}$. In this case, indeed, for each $\mu$ there is a $k\geq0$ such that the set of all periodic orbits for $\ell_\mu$ consists in a single orbit of period $2^i$ for $i=0,\dots,k$. 
\begin{definition}
    We say that two repelling nodes $N_1>N_2$ are {\bf consecutive} if there is no node $N_3$ such that $N_1>N_3>N_2$. 
%    , given a node $N_3$, when $N_1\geq N_3$ and $N_3\geq N_2$ then either $N_3=N_1$ or $N_3=N_2$. 
\end{definition}
\begin{proposition}
    \label{prop:k'=Ck}
    Let $N_1>N_2$ be two repelling nodes of $\ell_\mu$. 
    Then: 
    \begin{enumerate}
        \item  $|\cT(N_2)|>|\cT(N_1)|$ and $|\cT(N_1)|$ divides $|\cT(N_2)|$;
        \item $N_2\subset\Jall(N_1)\setminus\Jall(N_2)$;
        \item if $N_1$ and $N_2$ are consecutive, $N_2$ is the set of all chain-recurrent points in $\Jall(N_1)\setminus\Jall(N_2)$.
    \end{enumerate}
\end{proposition}
\begin{proof}
    {\bf 1.} 
%    It is sufficient to prove point (1) in case $N_1$ and $N_2$ are consecutive. 
    Set $\cT(N_1)=\{\J_1,\dots,\J_k\}$ and
    $\cT(N_2)=\{\J'_1,\dots,\J'_{k'}\}$.
    Since the maps $\ell_\mu:\J_i\to \J_{i+1}$,
    $i=2,\dots,k$ are all homeomorphisms, where we set $\J_{k+1}=\J_1$, then inside each $\J_i$ there must be the same number of intervals $\J'_i$, namely $k'=mk$ for some integer $m\geq1$.
    
    Suppose now that $k'=k$ and assume, for discussion sake, that $p_1(N_1)<c$.
    Then, since $\ell^k_\mu(\J'_1)\subset \J'_1$, both $p_1(N_1)$ and $p_1(N_2)$ are fixed points for $\ell^k_\mu$.

    If $p_1(N_2)<c$ as well, then $N_2$ would be an attracting node. Indeed, if it were not, there would be an attractor point $x$, belonging to an attracting node $N_3$, between $p_1(N_1)$ and $p_1(N_2)$. 
    Hence $N_2$ would be inside $\Jall(N_3)$ but, since $N_3$ is an attractor, $\Jall(N_3)$ is a subset of the basin of $N_3$.
%    , against the hypothesis that $N_1$ and $N_2$ are consecutive. 
    
    If, on the contrary, $p_1(N_2)>c$, then the derivative of $\ell^k_\mu$ at $p_1(N_2)$ would be negative and therefore $\J'_1$ cannot be invariant under $\ell^k_\mu$.
    
%    Hence, $p_1(N_2)$ is an attractor for $\ell^k_\mu$, against the hypotheses.
%    We claim that the case $m=1$ arises only when $N_2$ is an attractor. Indeed, if $m=1$, then $p_1(N_2)$ is a fixed point for $\ell^k_\mu$ and $\ell^k_\mu(c)<c$.
%    There are two cases:
%    CASE 1: $N_1,N_2$ are periodic orbits. In this case, we are in the middle of a bifurcation cascade and $|\cT(N_2)|=2|\cT(N_1)|$ (see Fig.~\ref{fig:p24ic}).

    {\bf 2.} Since $N_1>N_2$, there is at least one point of $N_2$ inside $\Jall(N_1)$ and therefore, since both $\Jall(N_1)$ and $N_2$ are forward-invariant, the whole $N_2$ must be contained in $\Jall(N_1)$. 
    Moreover, by construction, no point of $N_2$ can fall onto $\Jall(N_2)$.
    
    {\bf 3.} Suppose that there is a chain-recurrent point $x\in\Jall(N_1)\setminus\Jall(N_2)$ not belonging to $N_2$. Denote by $N_3$ the node $x$ belongs to.
    Without loss of generality we can assume that $x\in \J_1(N_1)$. 
    Then its trajectory under $\ell^k_\mu$ remains in $\J_1(N_1)$ and never enters $\J_1(N_2)$ by recurrence. 
    Hence the closest point of $N_3$ is bigger than $\rho(N_1)>\rho(N_3)>\rho(N_2)$. 
    %Then this point would belong to another node $N_3$ and, by the same argument above, $N_3$ would be a subset of $\Jall(N_1)\setminus\Jall(N_2)$
%    CASE 2: 
\end{proof}

There is an immediate important corollary of the proposition above.
\begin{corollary}
  For every repelling node $N$, there are only finitely many nodes $N'$ such that $N'>N$.
\end{corollary}

{
\begin{proposition}[{\bf Repelling Nodes}]
    \label{thm:repelling}
    Every repelling node of $\ell_\mu$ is either a flip periodic orbit or a Cantor set.
\end{proposition}
\begin{proof}
    Let $N_1,N_2,\dots$ be the set of nodes of $\ell_\mu$ ordered so that $\rho(N_i)>\rho(N_j)$ if $i<j$. 
    Then, $N_i$ is the set of chain-recurrent points in $\Jall(N_i)\setminus\Jall(N_{i+1})$.
    Denote by $\bar\ell$ the power $|\cT(N_i)|$ of $\ell_\mu$. We can always reduce the problem of the structure of the set $N_{i+1}$ to the problem of the structure of the set of points of the single interval $\J_1(N_i)$ not falling, under $\bar\ell$, on the cyclic trapping region $\cS=\{\J_1,\dots,\J_k\}$ of $\J_1(N_i)$ of period $k=|\cT(N_{i+1})|/|\cT(N_i)|$. Note that $\J_1=\J_1(N_{i+1})$.
    
    Assume first that
    $\cS$
%    $\cT(N_{i+1})=\{\J_1,\dots,\J_k\}$ 
    is regular. Then, in particular:
    \begin{enumerate}
        \item $\J_i\cap \J_j=\emptyset$ for $i\neq j$;
        \item for all $i\neq 1$, the set $\bar\ell^{-1}(\J_{i+1})$, where $k+1$ is meant mod $k$, is the disjoint union of $\J_i$ and a second interval $\hat \J_i$, on each of which $\ell_\mu$ restricts to a diffeomorphism with $\J_i$;
        \item $\bar\ell^{-1}(\J_{2})=\J_1$. In this case every point of $\J_1$ is covered by two points of $\J_2$ except for the critical point $c$. 
    \end{enumerate}
    Now, take any two intervals $A$ and $B$ that are connected components of, respectively, $\bar\ell^{-p}(\J_i)$ and $\bar\ell^{-q}(\J_j)$. Assume, to avoid trivial cases, that neither $A$ nor $B$ are equal to $\J_1$. 
    Then, after enough iterations of $\bar\ell$, say $r$, $\bar\ell^r(A)$ and $\bar\ell^r(B)$ will be subsets, respectively, of two intervals $\J_{i_A}$ and $\J_{i_B}$ and at least either $\bar\ell^r(A)$ or $\bar\ell^r(B)$ will be actually equal to that interval.
    
    If $i_A\neq i_B$, then $\bar\ell^r(A)\cap\bar\ell^r(B)=\emptyset$ and therefore also $A\cap B=\emptyset$. 
    If $i_A=i_B=i$, it means that $A$ and $B$ are backward iterates of the same $\J_i$. 
    In that case, after applying enough times $\bar\ell$, we will have that $\J_i\cap\bar\ell^{rk}(\J_i)\neq\emptyset$ for some positive integer $r$. 
    The only non-trivial case is when this intersection is a single point. In this case, the common point must be the periodic endpoint, but this can happen only if $\cS$ is a flip cyclic trapping region, against the hypothesis.
    
    % $\ell^r_\mu(A)\cap\ell^r_\mu(B)$ will be a neighborhood of the periodic endpoint of $\J_{i}$, namely $A$ and $B$ are one inside the other.
    % Finally, notice that the counterimages of the $\J_i$ must be dense in $\Jall(N_i)$ because almost all points converge to the attractor. 
    
    Ultimately, therefore, the set of points of $\J_1(N_i)$ that never fall in $\Jall(\cS)$ is the complement of a countable dense set of open intervals whose closures are pairwise disjoint.
    %such that the closure of any two distinct intervals have no common point. 
    Hence, it is a Cantor set and the action of $\bar\ell$ on it is a subshift of finite type. 
    The node $N_{i+1}$ is a closed invariant subset of a finite union of such Cantor sets, and therefore is itself a Cantor set.
    
    % Then, $A\cap B=\emptyset$ if $i-j\neq p-q \mod k$. Indeed, assume for discussion sake that $p\geq q$. 
    % Then
    % %, if $A\cap B\neq\emptyset$, also    $\ell_\mu^p(A)\cap\ell_\mu^p(B)\neq\emptyset$. By definition, though, 
    % $\ell_\mu^p(A)\subset \J_i$ and $\ell_\mu^p(B)\subset \ell_\mu^{p-q}(\J_j)\subset \J_{j'}$, where $j'=j+p-q \mod k$. 
    % Hence $\ell_\mu^p(A)\cap\ell_\mu^p(B)=\emptyset$, and therefore $A\cap B=\emptyset$, if $i\neq j+p-q \mod k$. 
    % Moreover, when $i=j+p-q \mod k$, then both $\ell_\mu^p(A)$ and $\ell_\mu^p(B)$ are contained in the same $\J_i$
    % % Moreover, when $i=j+p-q \mod k$, suppose that, for discussion sake, $i\geq j$. 
    % % Then, since $p=q+i-j+rk$ for some integer $r$, then
    % % $\ell^{-p}_\mu(\J_i)=\ell^{-p}_\mu(\ell_\mu^{i-j}\J_j)=\ell^{-p+i-j}_\mu(\J_j)=\ell^{-q-rk}_\mu(\J_j)=\ell^{-rk}_\mu(\ell_\mu^{-q} \J_j)$

    % If $\cT(N_{i+1})$ is regular, then $N_i$ is a closed subset of the complement of a dense set of open intervals with pairwise empty closure. Hence, $N_i$ is a Cantor set. 
    
    Assume now that $\cS$ is a flip cyclic trapping region. Then $\cS=\{\J_1,\J_2\}$, with $\J_1=\J_1(N_{i+1})$.
    Hence, $\cS$ is as in Fig.~\ref{fig:p24ic}, namely $\J_1=[q_1,p_1]$ and $\J_2=[p_1,q_2]$, where $p_1=p_1(N_{i+1})$ is fixed for $\bar\ell$.
    Then $\bar\ell^{-1}(\J_1)$ is the disjoint union of $\J_2$ and the interval $A=[\hat q_1,q_1]$, where $\hat q_1$ is the counterimage of $q_1$ at the left of $c$, while $\bar\ell^{-1}(\J_2)=\J_1$. The two counterimages of $A$ are the intervals $A_1=[\doublehat{q}_1,\hat q_1]$ and $A_2=[\bar q_1,\bar{\bar{q}}_1]$, where $\doublehat{q}_1$ is the counterimage of $\hat q_1$ at the left of $c$, $\bar{\bar{q}}_1$ the one at the right of $c$ and $\bar{q}_1=q_2$ is the counterimage of $q_1$ at the right of $c$. Similarly, at every new recursion step, two new intervals arise, one at the left of $c$ and having an endpoint in common with the interval at the left of $c$ obtained at the previous recursion level and one at the right with similar properties. 
    
    Ultimately, then, the set of points of $\J_1(N_{i})$ that do not fall eventually in $\Jall(N_{i+1})$ under $\bar\ell$ is the union of the fixed point $p_1(N_i)$ together with all of its counterimages under $\bar\ell$. These counterimages can be sorted into two subsequences which converge monotonically to the endpoints of $\J_1(N_i)$. Hence, in this case $N_{i+1}$ consists exactly in the flip periodic orbit through $p_1(N_{i+1})$.

\end{proof}
}
%\noindent{\bf Attracting Nodes.}
%

    The next proposition asserts that each attracting node $A$ of $\ell_\mu$ is one of the following five types.
    \begin{enumerate}[label=(\subscript{A}{{\arabic*}})]
        \item An attracting periodic orbit.
        \item A trapping region; here the attractor is a finite collection of intervals.
%        . It is always chaotic and has a dense trajectory;
        \item An attracting Cantor set; this is the attracting node when the number of nodes is infinite;
        \item A repelling Cantor set containing a 1-sided attracting periodic orbit.
        This occurs precisely for those $\mu$ for which there is a periodic orbit that is at the beginning of a window, where there is an attractor-repellor periodic orbit bifurcation. 
        That periodic orbit is a one-sided attractor and it is in a Cantor set.
        Fig.~\ref{fig:p3ic} gives an example: at $\pthree0$ there is a period-3 attractor-repellor bifurcation and $p_1$ is one of the three points of the attractor-repellor orbit, the one closest to the critical point. 
        Here there are two nodes: 0 and the red Cantor set, to which the orbit belongs to. 
        \item A trapping region which strictly contains an attracting cyclic trapping region, a repelling Cantor set and part of the basin of attraction. 
        This occurs precisely for those $\mu$ at the end of a window, where $c$ eventually maps onto the periodic orbit on the edge of a cyclic trapping region.
        Fig.~\ref{fig:p3ic} gives an example: at $\pthree5$ there is a crisis, the critical point eventually maps onto $p_1$. The line segment from $q_1$ to $p_1$ is $\J_1$. The node, however, is the whole interval from $q_3$ to $q_2$, which includes $\Jall$. Notice that, at such crisis values, $\ell_\mu(\J_1)=\J_2$, namely $q_2=\ell_\mu(c)$. In particular, $\ell^6_\mu(c)=p_1$.
%        \item a cyclic trapping region, namely an attracting trapping region and a repelling periodic orbit;
    \end{enumerate}
    %
    %Types 2 and three come each in two different flavors
\begin{proposition}[{\bf Attracting Nodes}]
    \label{thm:attracting}

The attracting node $A$ of $\ell_\mu$ is one of the above five types.
\end{proposition}
\begin{proof}
    The first three cases are those when an attracting node coincides with the attractor.
    We have case (1) for almost all $\mu\in\cA_P$, for instance for all $\mu\in(0,\mu_{FM})$.
    We have case (2) for almost all $\mu\in\cA_C$, for instance for $\mu=2(1+(19+3\sqrt{33})^{1/3}+(19-3\sqrt{33})^{1/3})/3$. At this parameter value, the critical point falls on the unstable fixed point in $(0,1)$ at its third iterate. This is the first $\mu$ for which the chaotic attractor is a single interval.
    We have case (3) for all $\mu\in\ALP$. 
%    For instance, we have case (1) for all $\mu\in(0,\mu_{FM})$; we have case (3) at $\mu=\mu_{FB}$; we have case (2) at $\mu=2(1+(19+3\sqrt{33})^{1/3}+(19-3\sqrt{33})^{1/3})/3$.

    Unlike the case of repelling nodes, though, in which case every two nodes have different distance from $c$ and so cannot ever merge, an attracting node can, in degenerate cases, merge with a repelling one. 
    
    When an attracting periodic orbit merges with a repelling one, we have an orbit that is attracting on one side and repelling on the other. This happens at every bifurcation point of a cascade. The node, though, in this case is still of type (1).
    
    When an attracting periodic orbit merges with a repelling Cantor set, we have a Cantor set whose points are all repelling except for a single periodic orbit, which is a 1-sided attractor. This is type (4).
%    n orbit that is attracting on one side and part of a repelling Cantor set on the other side. 
    We get this kind of node at the first point of each window, e.g. at $\mu=1+\sqrt{8}$ (see Fig.~\ref{fig:p3}).
    
    When an attracting trapping region merges with a repelling periodic orbit, the trapping region becomes a cyclic trapping region. We have this at $\mu=4$. The repelling node in this case is still of type (2).
    
    At the end of each window, the attractor is a trapping region with a periodic orbit in common with a repelling Cantor set.
    In this case the attractor is a cyclic trapping region and the node is equal to a trapping region which contains, besides the attractor, the Cantor set with the orbit in common with the attractor and part of the basin of attraction. We get this, for instance, at $\mu=3.8568...$ (see Fig.~\ref{fig:p3}).
%    merges with a repelling Cantor set, the node is equal to the whole cyclic trapping region of the window the corresponding value 
%    The same argument of the lemma above shows that a 2-sided attracting periodic orbit is a node in its own. 
%    In case the orbit is attracting from one side only, though, 
%    In case of an attracting Cantor set, all repelling nodes are periodic orbits and Cantor sets and so, again, the attracting node must be simply the attracting Cantor set.
%    In case the attractor is a trapping region, on the contrary, the attracting node can merge with either a periodic orbit or a Cantor set. 
\end{proof}
\begin{definition}
    \label{def:accessible}
    We say that a point $x$ in a node $N$ is {\bf accessible}~\cite{Bir32,GOY87,AY92} if there is a closed interval $K$ with $x$ as endpoint such that $K\cap N=\{x\}$. We call $K$ an {\bf access interval} of the node. A periodic orbit is accessible if all of its points are accessible.
\end{definition}
Notice that each periodic orbit is in some node. A node can either consist in a single periodic orbit or contain infinitely many. 
\begin{theorem}
    \label{thm:accessible}
    Let $N$ be a node of $\ell_\mu$ with $\rho(N)>0$. Then $N$ has a unique accessible periodic orbit in it. This accessible orbit is  $\orbit(N)$ and $\J_1(N)$ is an access interval of $N$.
\end{theorem}
\begin{proof}
    The claim is trivial when the node is a periodic orbit, so we assume that $N$ is a Cantor set. 
    By the same arguments used in the proof of Prop.~\ref{thm:repelling}, 
%    If $\hat N$ is the node such that $\hat N>N$ and the nodes are consecutive, then $$N\subset\overline{\Jall(\hat N)}.$$
    for each pair of points $x,y\in N$, $x<y$, between which lie no other point of $N$, there is some integer $r$ such that 
    $$\ell_\mu^{r}((x,y)) =\interior(\J_1(N)).$$
    In particular, $x$ and $y$ are preperiodic and fall eventually in $\orbit(N)$. One of them might be periodic.
\end{proof}
%The existence of a bijection between nodes and cyclic trapping regions has several interesting consequences:

%As a corollary of the proposition above, we have that every repelling node, besides the fixed point 0, lies within a pair of consecutive cyclic trapping regions: 
%
%\begin{corollary}
%  The number of cyclic trapping regions of any $\ell_\mu$ is at most countable.
%\end{corollary}

%
% Since there are points of $N$ on the boundary of $\cT_N$, we also have the following characterization:
% %
% \begin{proposition}
%   The cyclic trapping region $\cT_N$ is the largest one such that no point of $N$ falls in $\Jall(\cT)$ under $\ell_\mu$.
% \end{proposition}
%
%In the remainder of the section, we order the set of nodes $N_1,\dots,N_n$ so that $N_i>N_j$ for $j>i$.
%We allow the possibility $n=\infty$ when there are countably many nodes. 
%For each node $N_j$ we have a corresponding cyclic trapping region $\cT_j$ and $N_j> N_{j'}$ implies that $\Jall(\cT_j)\supset\Jall(\cT_{j'})$.

\bigskip\noindent{\bf Edges.} 
The following proposition 
%shows that the order relation induced on nodes by $\rho$ coincides with the one 
is the last non-trivial step we need in order to prove that the graph is a tower:
\begin{proposition}
    \label{prop:upstream}
    If $N$ and $N'$ are nodes and $N>N'$, then there is an edge from $N$ to $N'$.
%    If $N_1>N_2$, then $N_1$ is upstream from $N_2$.
%    Let $N_1,N_2$ be two nodes and $\cT_1,\cT_2$ their corresponding cyclic trapping regions. The only node in $\Jall(\cT_1)\setminus\Jall(\cT_2)$ is $N_2$.
\end{proposition}
\begin{proof}
    By construction, the node $N$ has no common point with $\Jall(\cT_{N})$. Since $N'$ is closer to $c$ than $N$, on the contrary, at least one of its points lies in the interior of $\J_1(N)$ and therefore the whole $N'$ lies in $\Jall(\cT_{N})$.

    By Prop.~\ref{prop:downstream}, each point $x\in N'$ is downstream from $p_1(N)$. Since $p_1(N)$ is periodic, for every $\varepsilon>0$ there is a trajectory $t_\varepsilon$ starting in $(p_1,p_1+\varepsilon)$ and falling eventually on $x$. Since $p_1$ is repelling, for each point $y$ close enough to $p_1$ there is a trajectory $t$ passing through $y$ with $\alpha(t)=N$. 
    In other words, there are trajectories backward asymptotic to $N$ from any node inside $\cT_{N}$, namely there is an edge from $N$ to any node in $\Jall(N)$. 

\end{proof}
\allblack
%Ultimately, all these properties lead to the following:
Now, the main result of our article comes immediately from Prop.~\ref{prop:rho} and Prop.~\ref{prop:upstream}.
\begin{theorem}
    \label{thm:tower}
%    For all $\mu\in(0,4]$, t
    The graph of $\ell_\mu$ is a tower. The tower is infinite if and only if $\mu\in\ALP$.
\end{theorem}
Several examples of towers are shown in Figs.~\ref{fig:full} and~\ref{fig:p3}. 
The white node on top is the node $N_0=\{0\}$. 
Its trapping region is the single interval $J_1=[0,1]$.
Each green node is a repelling periodic orbit, the red and blue nodes are repelling Cantor sets shown with the same color in figure. 
The black node is the attracting node. 
Notice that, in all towers in the interior of the period-3 window, the second node $N_1$ of the tower is the red Cantor set of all chain-recurrent points not falling eventually in $\Jall(N_1)$. 
In general, the number of repelling Cantor set nodes in the graph of $\ell_\mu$ is equal to the number of its regular repelling trapping regions, that is, the number of nested windows at $\mu$.
\allblack

\bigskip\noindent
{\bf Assigning weights to the edges of the graph.} As shown in Prop.~\ref{prop:k'=Ck}, the edge between $N_i$ to $N_{i+1}$ has a weight associated to it equal to $w_i=|\cT(N_{i+1})|/|\cT(N_{i})|$.
%Notice that, as a consequence of 
Recall that, in the logistic map, there are windows of any period $k\geq3$ and the structure of the bifurcation diagram is closely repeated in every subwindow. 
Hence, given any finite or infinite sequence of strictly increasing integers $s_n$ starting with $s_1=1$ and such that $s_n=w_n s_{n+1}$, 
%Then we claim that 
there is a $\mu$ in parameter space such that there is a node $N_n$ for each $n$ with $s_n=|\cT(N_n)|$. 
Furthermore, if $w_{n}=2$, then $\cT(N_n)$ is a flip trapping region and $N_n$ is a flip periodic orbit. 
If, instead, $w_{n}\neq2$, then $\cT(N_n)$ is a regular trapping region and $N_n$ it is either a Cantor set or the attractor.
}

\bigskip\noindent
{\bf Spectral Theorem.} 
%How about writing here a new 'spectral decomp. thm' putting together all our results?
Next statement collects all most important results we achieved into a ``chain-recurrent'' version of the Spectral Theorem.
\begin{theorem}[Chain-Recurrent Spectral Theorem]
    \label{thm:spectradec}
    Let $\mu\in(1,4)$ and denote by $N_0, N_1,\dots,N_p$, where $p$ is possibly infinite, the nodes of $\ell_\mu$ sorted so that $N_i>N_j$ if $i<j$.
    %$p+1$ be the number, possibly infinite, of the nodes $N_i$ of $\ell_\mu$. 
    Then:
    \begin{enumerate}
        % \item $\ell_\mu$ has exactly $p-1$ non-trivial cyclic trapping regions $\cT_1,\dots,\cT_{p-1}$ nested one in the other: 
        % $$
        % \text{cl}(\Jall(\cT_{j}))\subset\Jall(\cT_{j-1}),\;1\leq j< p,
        % $$
        % where, by convention, we take $\cT_0$ as the trivial period-1 cyclic trapping region with $\J_1=[0,1]$. (Prop.~\ref{thm:node2ic})
        \item the cyclic trapping regions $\cT_j=\cT(N_j)$, $0\leq j<p$, are nested in each other:
        $$
        \text{cl}(\Jall(\cT_{j}))\subset\Jall(\cT_{j-1}),\;1\leq j< p.
        $$
        % where, by convention, we take $\cT_0$ as the trivial period-1 cyclic trapping region with $\J_1=[0,1]$. (Prop.~\ref{thm:node2ic})
        \item Every point in $\Jall(\cT_j)$ falls eventually under $\ell_\mu$ either onto $N_j$ or onto $\Jall(\cT_{j+1})$ for all $0\leq j<p$. (Prop.~\ref{prop:k'=Ck})
%        \item Each $N_j$, $1\leq j<p$, is the set of chain-recurrent points in $\Jall(\cT_{j-1})$ that do not fall eventually in $\Jall(\cT_{j})$.
        %$\Jall(\cT_0)=(0,1)$.
        \item $\cR_{\ell_\mu}$ writes uniquely as the disjoint union of its nodes $N_0,N_1,\dots,N_p$. %(Defs.~\ref{def:equiv} and~\ref{def:node})
        \item $N_0=\{0\}$.
        \item Each $N_j$, $1\leq j<p$, is  %repelling and is 
        either a Cantor set or a periodic flip orbit (Prop.~\ref{thm:repelling}). If it is a Cantor set, the action of $\ell_\mu$ on it is a subshift of finite type with a dense orbit~\cite{vS81}. 
        \item Each $N_j$, $0\leq j<p$, is repelling and hyperbolic~\cite{vS81}.
%        and it is either a Cantor set, on , node 
        \item $N_p$ is the unique attracting node of $\ell_\mu$~\cite{Guc79} and it is one of the five types $A_1,\dots,A_5$. (Prop.~\ref{thm:attracting})
        \item $N_p$ is hyperbolic except when $\mu$ is at the endpoints of a window~\cite{vS81} (cases $A_4$ and $A_5$).
        %$N_p$ is hyperbolically attracting except when it is of type $A_4$ and $A_5$.
        %; $\ell_\mu$ is transitive on each node;
        % \item each repelling node $N_i$ is either a periodic orbit or a Cantor set; the action of $\ell_\mu$ on the Cantor set is a subshift of finite type;
        % \item $N_n$ is of one of the five types listed in Thm.~\ref{thm:attracting};
        \item In each neighborhood of $N_i$, for each $j\geq i$, there are points falling eventually into $N_j$. (Prop~\ref{prop:upstream})
        \item When $p=\infty$, the attracting node $N_\infty$ is a Cantor set on which $\ell_\mu$ acts on it as an adding machine~\cite{vS81}.
    \end{enumerate}
\end{theorem}
%
%We also have the following corollaries:

%\bigskip

%For every $\mu\in(1,4)$, 0 is the top node. Assume that $\mu$ belongs 

%Notice finally that all the theorems of this section hold not only for the logistic map but, with minimal changes, also for any S-unimodal map. An interesting question is which kind of graphs might correspond to S-multimodal maps. 

%\section{Appendix: some of the literature on the logistic map}
%
%\subsection{The non-wandering set of the logistic map}
%\subsection{Some of the literature on the logistic map}
%
%{\bf Bifurcation diagrams.} 
%This section recalls from the literature several results we use in proving our results. 

%

\section*{Acknowledgements}

The authors are grateful to Todd Drumm and Michael Jakobson for helpful conversations on the paper's topic. All calculations to produce the pictures in the present article were performed on the HPCC of the College of Arts and Sciences at Howard University with C++ code wrote by the first author.

\bibliography{refs}
\end{document}